\documentclass{article}
\usepackage{amsmath,amssymb,amsfonts,amsthm}
\usepackage{amsthm}
\theoremstyle{definition}
\newtheorem{example}{Example}[section]
\usepackage{graphicx}
\usepackage{geometry}
\usepackage{mathrsfs}
\usepackage{tikz-cd}
\usepackage[all]{xy}
\usepackage{amscd}

\newtheorem{definition}{Definition}[section]
\newtheorem{theorem}{Theorem}[section]

\newtheorem{lemma}{Lemma}[section]
\newtheorem{proposition}{Proposition}[section]
\newtheorem{remark}{Remark}[section]
\newcommand\be{\begin{array}}
\newcommand\en{\end{array}}
\newcommand\di{\displaystyle}

\newcommand\ga{\gamma}
\newcommand\ep{\emptyset}

\newcommand\si{\sigma}

\newcommand\pa{\partial}

\newcommand\va{\varepsilon}
\newcommand\la{\lambda}
\newcommand\Om{\Omega}
\newcommand\wi{\widetilde}
\newcommand\de{\delta}

\newcommand\gl{\geqslant}
\newcommand\ls{\leqslant}

\newcommand\ri{\rightarrow}
\newcommand\ti{\times}

\newcommand\bk{\backslash}
\newcommand\al{\alpha}

\newcommand\no{\eqno}
\newcommand\mb{\mathbb}
\newcommand\ov{\overline}
\begin{document}
\title{\bf   Abstract  Indefinite Problems in   Riesz  Spaces with Its Applications
\thanks{This paper is supported   by the National Natural Science Foundation of China (Grant No. 12371115).}
}
\author{Xian  Xu$^1$,  Baoxia Qin$^2$} \maketitle

\vspace{2mm} \noindent \begin{center}\ $^1$Department of Mathematics,
Jiangsu Normal University, Xuzhou,
\\Jiangsu,
221116,  P. R. China

$^2$School of Mathematics,
Qilu Normal  University, Jinan,
\\Shandong,
250013,  P. R. China\\

\end{center}
\begin{center}
\begin{minipage}{5in}
{\small {\bf Abstract}\quad This paper investigates the existence of critical points for functionals defined on a Hilbert space
$X$  which is continuously embedded into a Banach lattice $E$. A lattice decomposition of $E$
is constructed, which possesses both order disjointness and inner-product orthogonality. Accordingly, a corresponding decomposition of the Hilbert space
$X$  is obtained. Under this decomposition, the associated functional satisfies the energy collapse condition and order-preserving property on certain subspaces, while exhibiting coerciveness on others. By combining the descending flow invariant set method with Morse theory, we establish the existence of multiple critical points for the abstract indefinite problems. Finally, applications to elliptic boundary value problems are provided.
}
\end{minipage}
\end{center}

\section{Introduction}

Over the past four decades, nonlinear problems with indefinite nonlinearities-arising naturally from differential geometry, mathematical biology and other applied fields-have attracted extensive attention and undergone in-depth research; see \cite{s1,s2,s3,s4,s5,s6,s7,s8,s9,s10} for a comprehensive overview. As a typical example, Chang and Jiang \cite{s1} investigated the existence and multiplicity of positive, negative and sign-changing solutions for the following elliptic boundary value problem:
$$\left\{\be{ll}
-\Delta u = \lambda u + a_+(x)|u|^{q-1}u - a_-(x)|u|^{p-1}u + h(x,u) & \text{in } \Omega, \\
u = 0 & \text{on } \partial\Omega,
\en
\right.\no(1.1)$$
where $\Omega\subset\mathbb{R}^N$ is a bounded domain, $a_\pm : \overline{\Omega} \to \mathbb{R}$ are continuous functions, $h : \overline{\Omega} \times \mathbb{R} \to \mathbb{R}$ is a $C^1$-smooth function, and $\lambda$ is a real parameter. The following structural assumptions are imposed on the problem:
\begin{enumerate}
\item[(A1)] $a_\pm \gl 0$, $\overline{\Omega}_+ \cap \overline{\Omega}_- = \emptyset$ and $\Omega_+ \neq \emptyset$, where $\overline{\Omega}_\pm = \mathrm{supp}(a_\pm)$;
\item[(A2)] $1 < q < 2^* - 1 = \frac{N+2}{N-2}$ and $p > 1$;
\item[(A3)] There exists a constant $C > 0$ such that
\[
|h(x,\xi)| \le C(1 + |\xi|), \quad \forall \xi \in \mathbb{R}.
\]
\end{enumerate}
By employing the heat flow as a deformation tool, Chang and Jiang developed a Morse theory for solutions of problem (1.1), and thereby obtained a series of existence and multiplicity results for positive, negative and sign-changing solutions.

In recent years, the study of sign-changing solutions for various boundary value problems has become a central topic in nonlinear analysis. A key condition widely used in this context is \textit{the order-preserving property}. Nevertheless, for boundary value problems involving indefinite nonlinearities, direct verification of this property is notoriously difficult. To elaborate on this difficulty, as discussed in \cite{s3}, we consider the classical Dirichlet problem
$$\left\{\be{ll}
-\Delta u = f(x, u), & x \in \Omega, \\
u = 0, & x \in \partial\Omega.
\en\right.\no(1.2)
$$
Assume that the nonlinearity $f$ satisfies the one-sided Lipschitz condition
$$
\inf_{u \neq v,\, x \in \Omega} \frac{f(x, u) - f(x, v)}{u - v} > -k \quad \text{for some constant } k \gl 0.\no(1.3)
$$
Then the operator
\[
u \mapsto K(u) := (-\Delta + k)^{-1}\bigl(f(\cdot, u) + k u\bigr)
\]
is strongly order-preserving. The gradient vector field $\nabla\Phi$ associated with the functional $\Phi \in C^1(E)$ (with respect to a suitable inner product on $E := H^1_0(\Omega)$) takes the form $\nabla\Phi = \operatorname{Id} - K$. Consequently, one can construct a pseudo-gradient flow that leaves the positive cone $P^+ := \{u \in E : u \gl 0 \text{ a.e.}\}$ and negative cone $P^- := -P^+$ positively invariant. This invariant property is indispensable for finding critical points outside the union $P^+ \cup P^-$, as demonstrated in  works \cite{s19,s20,s23,s24,s25} on nodal solutions of (1.2) under condition (1.3).

This approach, however, fails for problems with indefinite nonlinearities, since condition (1.3) is inevitably violated. On the other hand, a common auxiliary condition for establishing sign-changing solutions is the \textit{energy collapse condition}: $\lim_{t\to+\infty} \Phi(tu)=-\infty$ for every nontrivial $u$, which is extremely difficult to verify for indefinite problems. To surmount these obstacles, Chang and Jiang \cite{s1} introduced a space decomposition technique: $E=E_1\oplus E_2$, where
$$E_1=H_0^1(\overline{\Omega_0\cup\Omega_-})\cap L_{a_-}^{p+1}(\overline{\Omega_0\cup\Omega_-}),$$
$$E_2=\big\{u\in H_0^1(\overline{\Omega_0\cup\Omega_+})\mid\Delta u(x)=0,\ \forall x\in \Omega_0\big\},$$
and
$$L_{a_-}^{p+1}(\Omega)=\Big\{u\in \mathcal{D}'(\Omega)\,\Big|\,\int_\Omega a_-(x)|u(x)|^{p+1}dx<\infty\Big\}.$$
This decomposition is governed by the lattice structure of $L_{a_-}^{p+1}(\Omega)$, as well as the disjointness and orthogonal additivity of the nonlinearity in problem (1.1).

Motivated by the methodology in [1], we introduce an abstract lattice framework and study a class of abstract indefinite problems in the setting of Riesz spaces. The theory of Riesz spaces (i.e., vector lattices) has been well developed in functional analysis with a solid theoretical foundation; see \cite{s12} for a systematic introduction. In contrast, applications of topological degree theory or critical point theory to nonlinear problems on Riesz spaces remain relatively scarce. For instance, Sun and Liu \cite{s15} combined lattice structures with topological degree methods and proved fixed point existence theorems for certain nonlinear operators. Sun and Xu \cite{s16,s17} integrated lattice theory with bifurcation analysis to investigate the global structure of solution sets for parameterized nonlinear operator equations. Perera and Schechter \cite{s18} applied critical point theory to the study of Fu\v{c}\'{i}k spectra and jumping nonlinear problems.

In this paper, we focus on the existence of critical points for functionals defined on a real Hilbert space $X$. To this end, we introduce a Riesz space $E$ into which $X$ is continuously embedded, and construct a lattice decomposition of $E$ that satisfies both order disjointness (from the lattice structure) and inner-product orthogonality. A corresponding orthogonal decomposition of the Hilbert space $X$ is then derived. Under this decomposition, the associated functional satisfies the energy collapse condition and order-preserving property on certain subspaces, while exhibiting coerciveness on complementary subspaces. By combining the descending flow invariant set method with Morse theory, we establish the existence of multiple critical points for the abstract indefinite problems. Finally, we apply our abstract results to elliptic boundary value problems to illustrate the applicability of the theoretical framework.

\section{Some Basic Results Concerning Riesz Spaces}
First let us recall some basic concepts and results on Riesz spaces in this section. Unless otherwise specified, all definitions and lemmas in this section can be found in the monograph \cite{s12} on Riesz spaces by Luxemburg W. A. J. and  Zaanen A. C.. For the convenience of readers who wish to further develop research on the applications of lattices in critical point theory, this section collects some basic knowledge about lattices, which goes slightly beyond what is required for the proof of the main theorem in this paper.

Let $E$ be a real vector space and let $\ls$ denote a partial order on $E$. The pair $(E, \ls)$ is called a \textbf{Riesz space} if the partial order $\ls$ is compatible with the algebraic structure of $E$, i.e.,
\begin{enumerate}
    \item For any $x, y, z\in E$ with $x\ls y$, it holds that $x+z\ls y+z$;
    \item For any $x\in E$ with $x\gl 0$ and any real scalar $r\gl 0$, we have $rx\gl 0$;
\end{enumerate}
and every two-point subset $\{x, y\}\subset E$ admits a supremum $\sup\{x,y\}=x\vee y$ in $E$.

For any $x\in E$, define
$$x^+=x\vee 0,\quad x^-=(-x)\vee 0,\quad |x|=x^++x^-,$$
which are called the \textit{positive part}, \textit{negative part} and \textit{modulus} of $x$, respectively. It is obvious that $x= x^+-x^-$.

\begin{lemma}
Let $E$ be a vector lattice, and $x, y \in E$. Then the following statements hold:
\begin{enumerate}
    \item $x^+ \gl 0$, $x^- \gl 0$, $x^+ = (-x)^-$, $x^- = (-x)^+$, and $|x| = |-x|$;
    \item $x = x^+ - x^-$ and $x^+ \wedge x^- = 0$;
    \item $|x| = x^+ + x^- \gl 0$;
    \item $0 \ls x^+ \ls |x|$ and $0 \ls x^- \ls |x|$;
    \item $-x^- \ls x \ls x^+$;
    \item $|ax| = |a||x|$ for all $a \in \mathbb{R}$;
    \item $x \ls y \iff x^+ \ls y^+$ and $x^- \gl y^-$;
    \item $|x + y| \ls |x| + |y|$;
    \item $\big||x| - |y|\big| \ls |x| + |y|$.
\end{enumerate}
\end{lemma}

\begin{lemma}[Birkhoff's Inequalities]
Let $x,y,u,v\in E$. Then
\[|x\vee y-u\vee v|\ls|x-u|+|y-v|,\]
\[|x\wedge y-u\wedge v|\ls|x-u|+|y-v|.\]
\end{lemma}

\begin{definition}
Let $E$ be a Riesz space.
\begin{enumerate}
    \item A linear subspace $L$ of $E$ is said to be an \textbf{order ideal} if for all $x\in L$, $y\in E$, $|y|\ls|x|$ implies $y\in L$.
    \item An order ideal $L\subseteq E$ is called a \textbf{band} if whenever $\emptyset\neq D\subset L$ and $\sup D = x_0$ exists in $E$, then $x_0\in L$.
\end{enumerate}
\end{definition}

\begin{definition}
Let $E$ be a Riesz space and $x, y\in E$. We say that $x$ and $y$ are \textbf{disjoint} if $|x|\wedge|y|=0$, denoted by $x\perp y$. For any subset $D\subseteq E$, its \textbf{disjoint complement} is defined as
$$D^d=\{x\in E: x\perp y,\ \forall\, y\in D\}.$$
Two subsets $D_1,D_2\subseteq E$ are said to be disjoint, written $D_1\perp D_2$, if $x\perp y$ for all $x\in D_1$ and $y\in D_2$.
\end{definition}

\begin{lemma}
The following properties hold for disjoint elements and subsets in a Riesz space $E$:
\begin{enumerate}
    \item If $x\perp y$ and $|z|\ls |y|$, then $z\perp  x$;
    \item If $x\perp y$ and $\alpha\in\mathbb{R}$, then $\alpha x\perp y$;
    \item If $x_1\perp y$ and $x_2\perp y$, then $(x_1+x_2)\perp y$;
    \item $x\perp y$ if and only if $x^+\perp y^+$ and $x^-\perp y^-$;
    \item For any $x\in E$, $x^+\perp x^-$;
    \item If nonempty subsets $D_1,D_2\subseteq E$ satisfy $D_1\perp D_2$, then $D_1\cap D_2$ is either empty or $\{0\}$;
    \item Suppose that $\sup D = x_0$ exists for $D\subseteq E$. If $y\perp x$ for all $x\in D$, then $y\perp x_0$;
    \item Any finite set of nonzero mutually disjoint elements in $E$ is linearly independent.
\end{enumerate}
\end{lemma}

\begin{lemma}
Let $E$ be a Riesz space and $D\subseteq E$ be nonempty. Then:
\begin{enumerate}
    \item $D^d$ is a band in $E$;
    \item $D\subseteq D^{dd}$, so the band generated by $D$ is contained in $D^{dd}$. Moreover, $D^d=D^{ddd}$ and $D^d \cap D^{dd} = \{0\}$, which yields the direct sum decomposition $D^d+D^{dd}=D^d\oplus D^{dd}$.
\end{enumerate}
\end{lemma}

\begin{lemma}
Let $E$ be a Riesz space.
\begin{enumerate}
    \item For any nonempty subset $D\subseteq E$, $D^{dd}$ is the largest subset of $E$ possessing the same disjoint complement as $D$;
    \item If nonempty subsets $D,F\subseteq E$ satisfy $D\perp F$, then $D^{dd}\perp F^{dd}$. In particular, the bands generated by $D$ and $F$ are disjoint.
\end{enumerate}
\end{lemma}

\begin{definition}
Let $E$ be a Riesz space, $F$ a real linear space, and $T: E\to F$ an operator.
\begin{enumerate}
    \item $T$ is called \textbf{orthogonally additive} if $T(x+y)=Tx+Ty$ whenever $x\perp y$ in $E$;
    \item $T$ is called a \textbf{local operator} if $Tx\perp y$ for all $x,y\in E$ with $x\perp y$.
\end{enumerate}
\end{definition}

It is easy to see that if $T$ is local and $x\perp y$, then $Tx\perp y$ and further $Tx\perp Ty$, i.e., local operators are disjointness-preserving.

\begin{definition}
Let $E$ be a Riesz space equipped with a norm $\|\cdot\|$. The norm is said to be a \textbf{Riesz norm} if $|x|\ls |y|$ implies $\|x\|\ls\|y\|$. A Riesz space endowed with a Riesz norm is called a \textbf{normed Riesz space}. A complete normed Riesz space is termed a \textbf{Banach lattice}.
\end{definition}

\begin{lemma}
Let $E$ be a Banach lattice. Then the following conclusions hold:
\begin{enumerate}
    \item If $\|x_n-x\|\to 0$ and $\|y_n-y\|\to 0$ as $n\to\infty$, then $\|x_n\vee y_n - x\vee y\|\to 0$ and $\|x_n \wedge y_n - x\wedge y\|\to 0$. In particular, $\|x_n^+-x^+\|\to0$, $\|x_n^--x^-\|\to0$ and $\||x_n|-|x|\|\to0$. Hence the mappings $x\mapsto x^+$, $x\mapsto x^-$ and $x\mapsto |x|$ are sequentially continuous on $E$.
    \item If $\|x_n-x\|\to0$ and $x_n\gl y$ for all $n\in\mathbb{N}$, then $x\gl y$. Consequently, $x_n\gl 0$ and $\|x_n-x\|\to0$ imply $x\gl 0$, which means the positive cone $E^+$ is closed in $E$.
    \item Let $D\subseteq E$. If $\|x_n-x\|\to0$ and $x_n\perp y$ for all $y\in D$ and all $n\in\mathbb{N}$, then $x\perp y$ for all $y\in D$. Therefore, every disjoint complement is a closed subset of $E$.
\end{enumerate}
\end{lemma}

We write $x=\bigsqcup_{i=1}^m x_i$ if $x=\sum_{i=1}^m x_i$ and $x_i\perp x_j$ for all $i\neq j$. For $x,y\in E$, we define the relation $x\sqsubseteq y$ to mean that $x$ is a fragment of $y$, i.e., $x\perp (y-x)$.

\begin{proposition}\cite{s14}
Let $E$ be a Riesz space and $x, y \in E$.
\begin{enumerate}
    \item If $x \sqsubseteq y$, then
    \begin{enumerate}
        \item $x^+ \sqsubseteq y^+$ and $x^- \sqsubseteq y^-$;
        \item $x^+ \ls y^+$ and $x^- \ls y^-$;
        \item $x^- \perp y^+$ and $x^+ \perp y^-$;
        \item $|x| \sqsubseteq |y|$.
    \end{enumerate}
    \item The relation $x \sqsubseteq y$ holds if and only if $x^+ \sqsubseteq y^+$ and $x^- \sqsubseteq y^-$.
    \end{enumerate}
\end{proposition}

\begin{proposition}
The fragment relation $\sqsubseteq$ is a partial order on any Riesz space $E$.
\end{proposition}

\begin{remark}For more results on the relation $\sqsubseteq$ one is  referred to \cite{s13,s14}. If $x=\bigsqcup_{i=1}^m x_i$, then each $x_i( i=1,2,\cdots, m)$ is called a fragment of $x$ (see \cite{s13,s14}). Given $x\in E$,  its decomposition is usually not unique. \end{remark}

\begin{example}
Let $\Omega\subset \mathbb{R}^N$ be a bounded domain. For $p\gl 1$, define the positive cone in $L^p(\Omega)$ by
$$P=\big\{f\in L^p(\Omega): f(x)\gl 0 \text{ a.e. } x\in \Omega\big\}.$$
Endow $L^p(\Omega)$ with the partial order:
$$f\ls g \iff f(x)\ls g(x) \text{ a.e. in } \Omega.$$
Then $L^p(\Omega)$ forms a real Riesz space. The standard $L^p$-norm
$$\|f\|_p=\left(\int_\Omega|f(x)|^p dx\right)^{\frac{1}{p}},\quad f\in L^p(\Omega)$$
is monotone with respect to the above partial order, so $L^p(\Omega)$ is a Banach lattice.

For any $f, g\in L^p(\Omega)$, we have
$$f\perp g \iff |\mathrm{supp}\,f\cap \mathrm{supp}\,g|=0.$$
\end{example}

\begin{example}
A function $f:\Omega\times \mathbb{R}^N\to \mathbb{R}$ is said to be a \textbf{Carath\'{e}odory function} if
\begin{enumerate}
    \item For almost every $x\in\Omega$, $\xi\mapsto f(x, \xi)$ is continuous on $\mathbb{R}^N$;
    \item For each fixed $\xi\in\mathbb{R}^N$, $x\mapsto f(x, \xi)$ is measurable on $\Omega$.
\end{enumerate}
Suppose that $f:\Omega\times \mathbb{R}^N\to \mathbb{R}$ is a Carath\'{e}odory function satisfying $f(x,0)=0$ a.e. $x\in \Om$ and
$$|f(x, \xi)|\ls b(x)+a|\xi|^{\frac{p_1}{p_2}} \quad \text{a.e. } x\in \Omega,\ \forall\,\xi\in\mathbb{R}^N.$$
Define the Nemytskii operator $\mathcal{F}(u)(x):=f(x,u(x))$ for $u\in L^{p_1}(\Omega)$ and $x\in \Omega$.

It can be verified that for all disjoint elements $u, v\in L^p(\Omega)$, we have $\mathcal{F}(u)\perp v$, $\mathcal{F}(u)\perp \mathcal{F}(v)$ and $\mathcal{F}(u+v)=\mathcal{F}(u)+\mathcal{F}(v)$. Hence the Nemytskii operator $\mathcal{F}$ is both local and orthogonally additive.
\end{example}

\begin{example} Let $\Om\subset \mb R^N$.
For subsets $A,B\subset \Omega$, we define $A\subseteq_1 B$ if $A\setminus B$ is a Lebesgure-null set. Let $0\neq f\in L^p(\Omega)$. Then
$$\{f\}^{dd}=\mathrm{span}\big\{g\in L^p(\Omega): \mathrm{supp}\, g\subseteq_1 \mathrm{supp}\, f\big\}.$$
\end{example}

\section{Main Results}
\noindent
Let $X$ and $E$ be real Hilbert spaces equipped with inner products $(\cdot,\cdot)$ and $\langle\cdot,\cdot\rangle$, respectively. Let $\|\cdot\|$ and $\|\cdot\|_1$ denote the norms on $X$ and $E$. Suppose $X\hookrightarrow E$, then there exists a constant $c_0>0$ such that
\(
\|x\|_1\ls c_0\|x\|\) for $x\in X$.

Let $P_1$ be a cone in $E$ inducing the partial order $\ls$ on $E$, i.e., $x\ls y$ if and only if $y-x\in P_1$. Obviously, $P=P_1\cap X$ is also a cone in $X$, which induces a partial order on $X$ denoted identically by $\ls$. The dual cone of $P_1$ is given by
\[
P_1^*=\{g\in E^*: \langle g, u\rangle\gl 0,\ \forall\,u\in P_1\}.
\]

We further assume that $(E,\ls)$ is a Banach lattice, and so the norm $\|\cdot\|_1$ is a Riesz norm, namely $|x|\ls |y|$ implies $\|x\|_1\ls \|y\|_1$.
By the definition above, we can define the \textbf{disjointness} of elements in $E$: for any $x, y \in E$,
\[
x \perp y \quad \text{if and only if} \quad |x| \wedge |y| = 0.
\]
At the same time, $(E, \langle \cdot, \cdot \rangle)$ is a Hilbert space, so we can define the \textbf{orthogonality} of two elements in $E$: for any $x, y \in E$,
\[
x \perp_1 y \quad \text{if and only if} \quad \langle x, y \rangle = 0.
\]
We say that disjointness and orthogonality in $E$ are \textbf{compatible} if
\[
x \perp y \implies x \perp_1 y, \quad \forall x, y \in E.
\]
Also, for any $x, y \in X \subset E$, we say that $x$ and $y$ are \textbf{orthogonal} if
\[
x \perp_0 y \quad \text{if and only if} \quad (x, y) = 0.
\]
Similarly, we say that disjointness and orthogonality in $X$ are \textbf{compatible} if
\[
x \perp y \implies x \perp_0 y, \quad \forall x, y \in X.
\]

\paragraph{Note.}
In the above setting, we do not assume that $X$ is necessarily a lattice  with respect to the ordering $\ls$. We only consider $X$ as a subset of $E$.
\medskip

 Let $m_1, m_2\in \mb N$. Set  $\Lambda_1=\{1,2,\cdots, m_1\}$ and $\Lambda_2=\{1,2,\cdots, m_2\}$. Let $a\in X\bk\{0\}$ such that \[a^+=\bigsqcup_{i=1}^{m_1} a^+_i, a^-=\bigsqcup_{j=1}^{m_2} a^-_j,\] where $a^+_i\neq 0$ for $i\in\Lambda_1$, $a_j^-\neq 0$ for $j\in\Lambda_2$. Let
  $$G^d_{i,+}:=\{u\in E: u\bot a^+_i\},\  G^{dd}_{i,+}:=\{u\in E: u\bot h, \forall h\in G^d_{i,+}\}$$
for $i\in\Lambda_1$,   and
 $$G^d_{j,-}:=\{u\in E: u\bot a^-_{j}\},\  G^{dd}_{j,-}:=\{u\in E: u\bot h, \forall h\in G^d_{j,-}\}$$
 for $j\in\Lambda_2$.
  Assume that
  \[E=\bigoplus\limits_{i=1}^{m_1}G^{dd}_{i,+}\oplus\bigoplus\limits_{j=1}^{m_2}G^{dd}_{j,-}.\no(3.1)\]
 Let $D^{dd}_{i,+}=G^{dd}_{i,+}\cap X$ for $i\in\Lambda_1$, and  $D^{dd}_{j,-}= G^{dd}_{j,-}\cap X$ for $j\in\Lambda_2$.

\begin{lemma}
For each $i\in \Lambda_1$ and $j\in\Lambda_2$, $D^d_{i,+}$, $D^{dd}_{i,+}$, $D^d_{j,-}$ and $D^{dd}_{j,-}$ are all real Hilbert subspaces of $X$. Moreover, all of theses $D^{dd}_{i,+}$ ($i\in \Lambda_1$) and $D^{dd}_{j,-}$ ($j\in\Lambda_2$) form a Hilbert subspace under the direct sum: \[\bigoplus\limits_{i=1}^{m_1}D^{dd}_{i,+}\oplus\bigoplus\limits_{j=1}^{m_2}D^{dd}_{j,-}.\]
\end{lemma}
\begin{proof}
By Lemma 2.3, $D^d_{i,+}$, $D^{dd}_{i,+}$, $D^d_{j,-}$ and $D^{dd}_{j,-}$ are real linear subspaces of $X$. Since $a_i^+\perp a_j^-$, combining Lemma 2.4 and Lemma 2.5 yields $D^{dd}_{i,+}\perp D^{dd}_{j,-}$, so that the sum $D^{dd}_{i,+}+D^{dd}_{j,-}$ is a direct sum $D^{dd}_{i,+}\oplus D^{dd}_{j,-}$. Similarly, for any $i\neq i'$, $D^{dd}_{i,+}+D^{dd}_{i',+}$ is a direct sum $D^{dd}_{i,+}\oplus D^{dd}_{i',+}$, and for any $j\neq j'$, $D^{dd}_{j,-}+D^{dd}_{j',-}$ is a direct sum $D^{dd}_{j,-}\oplus D^{dd}_{j',-}$.

We next prove that $D^d_{i,+}$, $D^{dd}_{i,+}$, $D^d_{j,-}$ and $D^{dd}_{j,-}$ are closed in $X$. Let $\{u_n\}\subset D^d_{i,+}$ be a sequence with $\|u_n-u\|\to 0$ as $n\to\infty$ for some $u\in X$. Then $|u_n|\wedge |a_i^+|=0$ for all $n$. By Lemma 2.1 and Birkhoff's inequality,
\[
\big||u|\wedge|a_i^+|\big|
=\big||u|\wedge|a_i^+|-|u_n|\wedge|a_i^+|\big|
\leqslant \big||u_n|-|u|\big|
\leqslant |u_n-u|.
\]
Since $E$ is a Banach lattice and $X$ is continuously embedded in $E$, there exists a constant $c_0>0$ such that
\[
\big\||u|\wedge|a_i^+|\big\|_1
\leqslant \|u_n-u\|_1
\leqslant c_0\|u_n-u\|\to 0,\quad n\to\infty.
\]
Thus $|u|\wedge|a_i^+|=0$, which implies $u\in D^d_{i,+}$. Hence $D^d_{i,+}$ is closed in $X$. The closedness of $D^{dd}_{i,+}$, $D^d_{j,-}$ and $D^{dd}_{j,-}$ can be shown analogously.

As closed linear subspaces of the real Hilbert space $X$, all these subspaces are themselves real Hilbert spaces.
\end{proof}
\medskip

  In this paper, we make the following  assumptions:
\begin{itemize}
 \item [(H\(_{1})\)]\ Let $a\in X$ such that (3.1) holds. Assume that the disjointness and orthogonality are \textit{compatible} in $X$  and $E$,
 \(J:X\to \mathbb R\) such that \(J^{\prime}\) has the form:
\[( J^{\prime}(u),v) = ( u,v)-\langle \mathbf{f}(u),v\rangle, \forall u, v\in X,\]
where \(\mathbf{f}:X\to E^*(=E)\) is continuous, local,  orthogonally additive;
\item [(H\(_2)\)]\ \ For each $i\in\Lambda_1$ and $u\in D^{dd}_{i,+}\bk\{0\}$, \[\lim\limits_{t\ri+\infty}J(tu)=-\infty\ \textit{\bf( Energy collapse condition in}\ D_{i,+}^{dd});\]
\item [(H\(_3)\)]\  For each $j\in\Lambda_2$, \(\inf_{u\in D_{j,-}^{dd}}J(u)<0\) and
$$\lim\limits_{u\in D^{dd}_{j,-}, \|u\|\ri+\infty}J(u)=+\infty\ \textit{\bf (Coercive condition in}\ D_{j,-}^{dd}).$$
\end{itemize}

\begin{remark}According to (3.1), the linear subspace $X$ of $E$ has a space decomposition of the form: \[X=\bigoplus\limits_{i=1}^{m_1}D^{dd}_{i,+}\oplus\bigoplus\limits_{j=1}^{m_2}D^{dd}_{j,-}.\]\end{remark}

\begin{remark}\  Since $\mathbf{f}$ is local, we have     $x\bot \mathbf{f}(y)$ for each $x\in G^d_{i,+}$ and $y\in G^{dd}_{i,+}$. This implies that  $\mathbf{f}(G^{dd}_{i,+})\subset G^{dd}_{i,+}$,  and so \( \mathbf{f}(D^{dd}_{i,+})\subset G^{dd}_{i,+}\)
 for each $i\in\Lambda_1$.
 \end{remark}

 Let $$F(u)=\int^1_0\langle \mathbf{f}(su), u\rangle ds\ \mbox{ for}\  u\in X.$$
 Then we have $$J(u)=\frac{1}{2}\|u\|^2-F(u)\ \text{for}\  u\in X.$$

\begin{lemma}
The mappings $J$ and $F$ are orthogonally additive.
\end{lemma}

\begin{proof}
Since $\mathbf{f}$ is orthogonally additive and local, it follows that $F$ is orthogonally additive. Indeed, for any $u,v\in X$ with $u\perp v$ and $s\in[0,1]$, we have $\langle \mathbf{f}(s u),v\rangle=\langle  u, \mathbf{f}(s v)\rangle=0$. Consequently,
\[
F(u+v)=\int_0^1\big\langle \mathbf{f}\big(s(u+v)\big),u+v\big\rangle ds
=\int_0^1\langle \mathbf{f}(su),u\rangle ds+\int_0^1\langle \mathbf{f}(sv),v\rangle ds=F(u)+F(v).
\]

Note that the disjointness and orthogonality in $X$  is compatible. Then, we have $(u,v)=0$ for every $u,v\in X$ with $u\bot v$. Consequently, we have $$\|u+v\|^2=\|u\|^2+\|v\|^2\ \mbox{ for every}\  u,v\in X\  \mbox{with}\ u\bot v.$$ Now the orthogonal additivity of $J$ can be deduced easily.
\end{proof}

By the orthogonally additivity of $J$, we have $J(0)=2J(0)$, and so $J(0)=0$. From Lemma 3.1 and the Riesz Representation Theorem, there exists an operator \(A_i: D_{i, +}^{dd}\to D_{i, +}^{dd}\) such that for any \(u \in D_{i, +}^{dd}\),
\[
\left( A_i(u), w \right) = \langle \mathbf{f}(u), w\rangle,\quad \forall w \in D_{i, +}^{dd}.\tag{3.2}
\]
Furthermore, since the mapping \(\mathbf{f}: X \to E\) is continuous, the operator \(A_i: D_{i, +}^{dd} \to D_{i, +}^{dd}\) is continuous as well.
Similarly,
by using the Riesz Representation Theorem again, there exists an operator \(A: X\to X\) such that for any \(u \in X\),
\[\left( A(u), w \right) = \langle \mathbf{f}(u), w\rangle,\quad \forall w \in X.\]
By the uniqueness of $A$, we have $A_i=A|_{D_{i, +}^{dd}}$.  Moreover, we have for each $u\in X$,
\[\be{ll}\|J'(u)\|&=\sup\limits_{w\in X, \|w\|=1}(J'(u), w)=\sup\limits_{w\in X, \|w\|=1}[(u,w)-\langle \mathbf{f}(u), w\rangle]\\
&=\sup\limits_{w\in X, \|w\|=1}(u-A(u), w)=\|u-A(u)\|.\en\]
\medskip

Let \(K = \{u\in X: J'(u) = 0\}\), $K_c=\{u\in K: J(u)=c\}$ and $J^c:=\{u\in X: J(u)\ls c\}$ for $c\in \mb R$.

\begin{remark}  Obviously, from (H$_1$), \(u\in K\)  if and only if \[ ( u,v)=\langle \mathbf{f}(u),v\rangle, \forall v\in X.\]
\end{remark}

For an isolated critical point $u\in K$ with $J(u)=c$ we define the critical groups of $J$ at $u$ by
\[ C_q(J,u):=H_q(J^c\cap U, J^c\cap U\bk\{u\}), q\in \mb Z,\]
where $U$ is an open neighborhood of $u$ in $X$ and $H_*$ is singular homology groups (with coefficients  in an arbitrary group; see $\mb Q$).

\begin{definition}\cite{s26}\ Assume that $K\cap J^c=\ep$ for some $c\in \mb R$. Define
$C_q(J,\infty):=H_q (X, J^c)$.
\end {definition}

\begin{lemma}\cite{s18}\ If $C_q(J,0)\neq C_q(J,\infty)$ for some integer $q$ and $J$ satisfies the (P.S.) condition, then $J$ has a critical point $u\neq0$.
\end{lemma}

\begin{theorem} Let $G(u)=\langle \mathbf{f}(u),u\rangle-2F(u)$ for all $u\in X$.
 Assume that $J$ satisfies  $(\mathrm{H}_{1})$, $G(u)$ is bounded from below in each $D_{i.+}^{dd}$.  Then the following conclusions hold:
\begin{itemize}
\item[1)] If dim $D_{i,+}^{dd}=\infty$, and 0 is a minimal value point of $J$, $J$ satisfies condition $(\mathrm{H}_{2})$ and  the (P.S.) condition on each $D_{i,+}^{dd}$,  then $J$ possesses at least $2^{m_1}-1$ distinct critical points;

\item[2)]If $J$ satisfies condition $(\mathrm{H}_{3})$ and  the (P.S.) condition on each $D_{j,-}^{dd}$, then $J$ has at least $2^{m_2}-1$ distinct critical points;

\item[3)] If the conditions in 1) and 2) hold, then $J$ admits at least $2^{m_1+m_2}-1$ distinct nontrivial critical points.
\end{itemize}
\end{theorem}

\begin{proof}\ It follows from Lemma 3.2 that $F$ is orthogonally additive, and so $F(0)=2F(0)$. Consequently, we have $F(0)=0$ and $G(0)=0$.

First we assume that the conditions in 1) holds. To show the conclusion 1), we will follow some ideas in \cite{s18, s27}.  Let $i\in\{i,2,\cdots, m_1\}$ be fixed. For $u\in D_{i, +}^{dd}$ and $t>0$, we have
$$\be{ll}\di\frac{d}{dt}\big(J(tu)\big)&=t\|u\|^2-\big\langle \mathbf{f}(tu),u\rangle\\
&=\di\frac{1}{t}\big(2J(tu)-G(tu)\big)\\
&\ls\di\frac{1}{t}\big(2J(tu)-\inf\limits_{u\in D_{i,+}^{dd}} G(u)\big)\en\no(3.3)$$
and hence all critical value of $J|_{D_{i,+}^{dd}}$ are greater than or equal to $\frac{1}{2}\inf\limits_{u\in D_{i,+}^{dd}} G(u)$ (note that $\inf\limits_{u\in D_{i,+}^{dd}} G(u)\ls 0$ since $G(0)=0)$.  Take  $c<\frac{1}{2}\inf\limits_{u\in D_{i,+}^{dd}} G(u)$. Now we will show that $J^c\cap D_{i,+}^{dd}$ is a strong deformation retract of $D_{i,+}^{dd}$.

Let $S=\{u\in X: \|u\|=1\}$ and $S^+_i:=D_{i,+}^{dd}\cap S$. For $u\in S^+_i$ and $t>0$, we have $J(tu)\ri-\infty$ for $t\ri+\infty$, so $J(tu)\ls c$ for all sufficiently large $t$. By (3.3), we have $$J(tu)\ls c\Longrightarrow \frac{d}{dt}J(tu)<0.\no(3.4)$$
Thus, there is a unique $T_c(u)>0$ such that
$$t>(\mbox{resp.}\ =, <) T_c(u)\Longrightarrow J(tu)<(\mbox{resp.}\ =, >)c$$
and the map $T_c: S^+_i\ri (0, \infty)$ is $C^1$ by the implicit function theorem. Then
$$J^c\cap D_{i,+}^{dd} =\{tu\in D_{i,+}^{dd}: u\in S^+_i, t\gl T_c(u)\}$$
and the map $\eta: D_{i, +}^{dd}\times [0,1]\ri D_{i, +}^{dd}$:
$$\eta(u, t)=\left\{\be{ll}(1-t)u+tT_c(\pi(u))\pi(u),&u\in D_{i,+}^{dd}\bk J^c,\\
u,& u\in J^c,\en\right.
$$
where $\pi$ is the radial projection onto $S^+_i$, is a strong deformation retract  of $D_{i,+}^{dd}$ onto $J^c\cap D_{i,+}^{dd}$. So,
$C_0(J, \infty)=0$. Since $0$ is a minimal value point of $J|_{D_{i,+}^{dd}}$, we have $C_0(J,0)=\mb Q$, and thus $C_0(J,\infty)\neq C_0(J, 0)$. According to Lemma 3.3, $J|_{D_{i,+}^{dd}}$
has at least a critical point $ x_i$ in $D_{i,+}^{dd}$, and so
\[
( x_i, w) = \langle \mathbf{f}(x_i), w\rangle,\  \forall w \in D_{i, +}^{dd}.
\]
For $v\in X$, we may assume that $$v=\left(\bigsqcup_{k=1}^{m_1} v_{+,k}\right)\bigsqcup\left(\bigsqcup_{j=1}^{m_2} v_{-,j}\right)$$ according to (H$_0)$, where $v_{+,k}\in D_{k,+}^{dd}$ and $v_{-,j}\in D_{j,-}^{dd}$. Then we have
\[\be{ll}
( x_i, v) &=\left( x_i, (\bigsqcup_{k=1}^{m_1} v_{+,k})\bigsqcup(\bigsqcup_{j=1}^{m_2} v_{-,j})\right)\\
&=\sum\limits_{k=1}^{m_1}( x_i, v_{+,k} )+\sum\limits_{j=1}^{m_2}( x_i, v_{-,j} )\\
&=( x_i, v_{+,i} )=\langle \mathbf{f}(x_i), v_{+,i}\rangle\\
&=\sum\limits_{k=1}^{m_1}\langle \mathbf{f}(x_i), v_{+,k} \rangle+\sum\limits_{j=1}^{m_2}\langle  \mathbf{f}(x_i), v_{-,j} \rangle\\
&=\langle \mathbf{f}(x_i), v\rangle.\en\no(3.5)
\]
This implies that $x_i\in K$ for $i\in\Lambda_1$. Let $z_{i,j}=x_i+x_j$ for $i,j\in \Lambda_1$ with $i\neq j$, where $x_i$ and $x_j$ are given in (3.5). It follows from (3.5) and (3.6) that  for every $v\in X$,
\[\be{ll}
( z_{i,j}, v)& =( x_i+x_j, v)=\langle \mathbf{f}(x_i)+\mathbf{f}(x_j), v\rangle\\
&=\langle \mathbf{f}(x_i+x_j),v\rangle=\langle \mathbf{f}(z_{i,j}),v\rangle.\en
\]
This implies that $z_{i,j}\in K$ for $i,j\in \Lambda_1$ with $i\neq j$.  In this way we can show that $J$ admits at least $2^{m_1}-1$ distinct nontrivial critical points such that each of them has the form
$\bigsqcup\limits_{i=1}^{m_1}(b_i x_{i})$, where $b_i\in \{0,1\}$ for $i\in\Lambda_1$  with $0\neq (b_1, b_2,\cdots, b_{m_1})\in \mb R^{m_1}$.

Now we assume the conditions in  2) hold. Then $J$ is coercive on $D_{j,-}^{dd}$. Since $J$ satisfies the (P.S.) condition on $D_{j,-}^{dd}$, $J$ has a minimal value point $y_j$ for $j\in\Lambda_2$, that is $J(y_j)=\inf\limits_{x\in D_{j,-}^{dd}}J(x)<0$. So, $y_j\neq 0$. Then we have
\[
( y_j, w) = \langle \mathbf{f}(y_j), w\rangle,\  \forall w \in D_{j, -}^{dd}.
\]
In a similar way as in (3.5), we have
\[
( y_j, v) = \langle \mathbf{f}(y_j), v\rangle,\  \forall v\in X.\no(3.6)
\]
This implies that $y_j\in K$ for $j\in\Lambda_2$. And so $J$ admits at least $2^{m_2}-1$ distinct nontrivial critical points such that each of them has the form
$\bigsqcup\limits_{j=1}^{m_2}(d_j y_{j})$, where  $d_j\in \{0,1\}$ for $j\in\Lambda_2$  with $0\neq (d_1, d_2,\cdots, d_{m_2})\in \mb R^{m_2}$.

At last we  assume that the conditions in 1) and 2) hold.  In a similar way as (3.5) we can show that $J$ admits at least $2^{m_1+m_2}-1$ distinct nontrivial critical points each of them has the form
\[\left(\bigsqcup\limits_{i=1}^{m_1}(b_i x_{i})\right)\bigsqcup\left(\bigsqcup\limits_{j=1}^{m_2}(d_jy_{j})\right),\]
where $b_i\in \{0,1\}$ for $i\in\Lambda_1$, $d_j\in \{0,1\}$ for $j\in\Lambda_2$  with $0\neq (b_1, b_2,\cdots, b_{m_1},d_1, d_2,\cdots, d_{m_2} )\in \mb R^{m_1+m_2}$.
\end{proof}

\begin{remark}In the PDE setting, $\left(\bigsqcup\limits_{i=1}^{m_1}(b_i x_{i})\right)\bigsqcup\left(\bigsqcup\limits_{j=1}^{m_2}(d_jy_{j})\right)
$ always be called a $s+t$-bump solution, where $s=\sum_{i=1}^{m_1}b_i$ and $t=\sum_{j=1}^{m_2}d_j$.\end{remark}

 Assume that there exists another Banach space $Z$ with the norm $\|\cdot\|_2$ such that $X\hookrightarrow Z\hookrightarrow E$. Assume that $\|u\|_2\ls c_1\|u\|$ for each $u\in X$ and some $c_1>0$. Let $P_0=P_1\cap Z$. Then $P_0$ is a cone in $Z$ which introduces an ordering $\ls$ in $Z$. Assume that $(Z,\ls)$ is a Banach lattice.

For any \(u\in K\setminus\{0\}\), we call \(u\) a positive critical point of $J$ when \(u\in P\), a negative critical point of $J$ when \(u\in -P\), and a sign-changing critical point of $J$ when \(u\notin P\cup(-P)\). Next we consider the existence of positive, negative and sign-changing critical points of $J$. For this purpose we introduce the following conditions:
\medskip
\begin{itemize}
\item[(H\(_4)\)]\  $P\cap D_{i,+}^{dd}\neq\ep$ and $\mathbf{f}(P\cap D_{i,+}^{dd})\subset P^*_1$ for each $i\in\Lambda_1$;

\item[(H\(_5)\)]\   There exists $g: \mb R_+\ri \mb R_+$ such that $ \langle \mathbf{f}(u), w\rangle\ls g(\|u\|_2)\|w\|_2$ for each $u,w\in D_{i,+}^{dd}$, where  $g(s)$ is nondecreasing in $s\in \mb R_+$ and $g(s)=o(s)$ as $s\ri 0^+$.
\end{itemize}
 \begin{remark} That $\mathbf{f}(P\cap D_{i,+}^{dd})\subset P^*_1$ is an order-preserving condition in $D_{i,+}^{dd}$.\end{remark}

\begin{theorem}\ Assume that $J$ satisfies the conditions (H\(_{1})\), (H\(_4)\) and (H\(_5)\). Then we have
\begin{itemize}
 \item [1)]\ If $J$ satisfies  condition (H$_2)$ and  the (P.S.) condition on each $D_{i,+}^{dd}$,  then $J$ has at least $2^{2m_1}-1$ nontrivial critical points. Among them there at least $2^{m_1}-1$ positive critical points, $2^{m_1}-1$ negative  critical points and ${m_1}$ sign-changing critical points;

     Moreover, if $J$ is of $C^2$,  $m_1=1$, $m_2=0$ and $K\bk(P\cup(-P))$ is a finite number set, then $J$ has at least one positive critical point $x_1$, one negative  critical point $y_1$ and one sign-changing critical point $z_1$ with $C_2(J, z_1)\neq 0$;

\item[2)] If $J$ satisfies condition $(\mathrm{H}_{3})$ and  the (P.S.) condition on each $D_{j,-}^{dd}$,  then $J$ has at least $2^{m_2}-1$ distinct critical points;

\item[3)] If $J$ satisfies both $(\mathrm{H}_{2})$ and $(\mathrm{H}_{3})$ and  the (P.S.) condition on each $D_{i,+}^{dd}$ and $D_{j,-}^{dd}$, then $J$ admits at least $2^{2m_1+m_2}-1$ critical points in total. Among them, there exist at least $2^{m_1}-1$ positive critical points, $2^{m_1}-1$ negative critical points and ${m_1}$ sign-changing critical points.
\end{itemize}
\end{theorem}

 Until the proof of Theorem 3.2 below, we will let $i\in\Lambda_1$ be fixed.

\begin{lemma}\ Assume that  (H\(_{1})\),  (H\(_{4})\) and (H\(_{5})\) hold. For any given \(\mu> 0\), denote \[D^{\pm}_i(\mu)=\{x \in D^{dd}_{i,+}: \mbox{dist}\ \big(x, (\pm P)\cap D_{i, +}^{dd}\big)<\mu\}.\] There exists \(\mu_{0} > 0\) such that for any \(0 < \mu < \mu_{0}\), we have \[A_i\left(\mbox{Cl}_{D_{+,i}^{dd}}D_i^{\pm}(\mu)\right) \subset D_i^{\pm}(\frac{1}{2}\mu),\] where $\mbox{Cl}_{D_{i,+}^{dd}}D_i^{\pm}(\mu)$ denotes the closure of $D_i^{\pm}(\mu)$ in $D_{i,+}^{dd}$.
\end{lemma}

\begin{proof}\ For each $u\in E$, we write $u_+=u^+$ and $u_-=-u^-$.
Let $v=A_i(u)$ for each $u\in D_{i,+}^{dd}$.  Since $\mathbf{f}(P\cap D_{i,+}^{dd})\subset P^*_1$ and $v_-\in (-P)\subset (-P_1)$, if $u\in P\cap D_{i,+}^{dd}$ we have
\[\|v_-\|^2=( v_-, v_-)=( v, v_-) = \langle \mathbf{f}(u), v_-\rangle\ls 0,\]
and so $v_-=0$, $v=v_+\in P$. This implies that $A_i(P\cap D_{i,+}^{dd})\subset P\cap D_{i,+}^{dd}$.

By (H$_1)$ and (H$_5)$ we have for any $u\in \mbox{Cl}_{D_{i,+}^{dd}}D^+_i(\mu)$,
\begin{align*}
\|v_-\|^2&=( v_-, v_-)=( v, v_-)\\
  &= \langle \mathbf{f}(u), v_-\rangle=\langle \mathbf{f}(u_+)+\mathbf{f}(u_-), v_-\rangle\\
  &\ls \langle \mathbf{f}(u_-), v_-\rangle\ls  g(\|u_-\|_2)\|v_-\|_2\\
  &\ls  c_1g(\|u_-\|_2)\|v_-\|,
\end{align*}
and so $\|v_-\|\ls c_1g(\|u_-\|_2)$.  Since $v=v_++v_-$ and $v_+\in P\cap D_{i, +}^{dd}$, we have
\[\mbox{dist}\ (v, P)\ls \|v_-\|\ls c_1 g(\|u_-\|_2).\]
For each $w\in P$, we have $u-w\ls u$. It follows from Lemma 2.1 that $$|u_-|=u^-\ls (u-w)^-\ls |u-w|.$$   Since  $\|\cdot\|_2$ is a Riesz norm and $g: \mb R_+\ri \mb R_+$ is nondecreasing, we have
\[\mbox{dist}\ (v, P)\ls c_1g(\|u_-\|_2)\ls c_1 g(\|u-w\|_2)\ls c_1g(c_1\|u-w\|).\]
Consequently, we have $\mbox{dist}\ (v, P)\ls c_1 g(c_1\mbox{dist}\ (u, P))$. Since $\lim\limits_{s\ri 0^+}\frac{g(s)}{s}=0$, there exists \(\mu_{0}^+ > 0\) such that for any \(0 < \mu < \mu_{0}^+\), we have  \(A_i\left(\mbox{Cl}_{D_{i,+}^{dd}}D_i^{+}(\mu)\right) \subset D_i^{+}(\frac{1}{2}\mu)\).

Similarly, there exists \(\mu_{0}^- > 0\) such that for any \(0 < \mu < \mu_{0}^-\), we have \(A_i\left(\mbox{Cl}_{D_{i,+}^{dd}}D_i^{-}(\mu)\right)\subset D_i^{-}(\frac{1}{2}\mu)\). Let $\mu_0=\min\{\mu^+_0,\mu^-_0\}$. Then the conclusion holds.
\end{proof}

Let $\mu\in (0,\mu_0)$ be fixed. For brevity, we write $D_1=D^+_i(\mu)$ and $D_2=D_i^{-}(\mu)$.

\begin{lemma}\ Assume that  (H\(_{1})\),  (H\(_{4})\) and (H\(_{5})\) hold.  Denote \(\widetilde{D}_i=D^{dd}_{i,+}\backslash K\). There exists a locally Lipschitz mapping \(B:\widetilde{D}_i\to D^{dd}_{i,+}\) satisfying

(1)\  For any given \(\mu\in(0,\mu_{0})\), we have \(B\left(\mbox{Cl}_{D_{+,i}^{dd}}D_i^{\pm}(\mu)\setminus K\right) \subset D_i^{\pm}(\frac{1}{2}\mu)\);

(2)\ \(\frac{1}{2}\|u - A_i(u)\|\ls\|u - B(u)\|\ls 2\|u - A_i(u)\|, \forall u\in\widetilde{D}_i\);

(3)\ \[( J^{\prime}(u),u - B(u))\gl\frac{1}{2}\|u - A_i(u)\|^{2}, \forall u\in \widetilde{D}_i.\tag{3.7}\]
\end{lemma}

\begin{proof}\ For all $u \in \widetilde{D}_i$, $u \neq A_i(u)$, and thus $\|J'(u)\|=\|u - A_i(u)\| > 0$. We give the following definitions:
\[
\Delta_1(u) = \frac{1}{2}\|u - A_i(u)\| > 0.
\tag{3.8}
\]
Take $r(u) \in (0,1)$ such that for all $v, w \in N(u) = \{v \in \widetilde{D}_i : \|v - u\| < r(u)\}$,
\[
\|A_i(v) - A_i(w)\| < \min\left\{\Delta_1(v), \Delta_1(w)\right\}
\tag{3.9}
\]
holds.

By the paracompactness of the space $\widetilde{D}_i$, let $\mathcal{M}$ be a locally finite open cover of $\{N(u) : u \in \widetilde{D}_i\}$. Now we construct the required operator $B$ by further refining the open cover $\mathcal{M}$. For each $U \in \mathcal{M}$, define:
\[
\mathcal{M}^* = \{U \in \mathcal{M} : U \cap \ov{D}_1^{D_{i,+}^{dd}} \neq \emptyset,\ U \cap \ov{D}_2^{D_{i,+}^{dd}} \neq \emptyset,\ U \cap \ov{D}_1^{D_{i,+}^{dd}} \cap \ov{D}_2^{D_{i,+}^{dd}} = \emptyset\}.
\]
We define
\[
\mathcal{N} := \bigcup_{U \in \mathcal{M}\bk \mathcal M^*} \{U\} \cup \bigcup_{U \in \mathcal{M}^*} \{U \setminus \ov{D}_1^{D_{i,+}^{dd}}, U \setminus \ov{D}_2^{D_{i,+}^{dd}}\}.
\]
Then $\mathcal{N}$ is also a locally finite open refinement of the covering $\{N(u) : u \in \widetilde{D}_i\}$ of $\widetilde{D}_i$. Moreover, if $U \in \mathcal{N}$ satisfies $U \cap \ov{D}_1^{D_{i,+}^{dd}} \neq \emptyset$ and $U \cap \ov{D}_2^{D_{i,+}^{dd}} \neq \emptyset$, then necessarily $U \cap \ov{D}_1^{D_{i,+}^{dd}} \cap \ov{D}_2^{D_{i,+}^{dd}} \neq \emptyset$.

We now construct the operator $B$. Let $\{\beta_U : U \in \mathcal{N}\}$ be a partition of unity subordinate to $\mathcal{N}$, defined by
\[
\beta_U(u) = \left( \sum_{V \in \mathcal{N}} \alpha_V(u) \right)^{-1} \cdot \alpha_U(u), \quad \forall u \in \widetilde{D}_i,\ U \in \mathcal{N},
\]
where $\alpha_U(u) := \operatorname{dist}(u, D^{dd}_{i,+} \setminus U)$. Clearly, for all $u, v \in D^{dd}_{i,+}$, we have $\|\alpha_U(u) - \alpha_U(v)\| \ls \|u - v\|$. From the definition of $\beta_U(u)$, it follows that
\[
0 \ls \beta_U(u) \ls 1, \quad \sum_{U \in \mathcal{N}} \beta_U(u) = 1.
\]
Combined with the local finiteness of $\mathcal{N}$, this implies that $\beta_U(u)$ satisfies a local Lipschitz condition on $\widetilde{D}_i$.

For each $U \in \mathcal{N}$, we choose an appropriate point $a_U \in U$ as follows:
\begin{itemize}
    \item If $U \cap \ov{D}_1^{D_{i,+}^{dd}} \neq \emptyset$, take $a_U \in U \cap \ov{D}_1^{D_{i,+}^{dd}}$;
    \item If $U \cap \ov{D}_2^{D_{i,+}^{dd}} \neq \emptyset$, take $a_U \in U \cap \ov{D}_2^{D_{i,+}^{dd}}$;
    \item If $U \cap \ov{D}_1^{D_{i,+}^{dd}} \cap \ov{D}_2^{D_{i,+}^{dd}} \neq \emptyset$, take $a_U \in U \cap \ov{D}_1^{D_{i,+}^{dd}} \cap \ov{D}_2^{D_{i,+}^{dd}}$;
    \item Other cases take any  $a_U \in D_{i,+}^{dd}\cap U$.
\end{itemize}

We define the operator $B : \widetilde{D}_i \to D_{i,+}^{dd}$ by
\[
B(u) = \sum_{U \in \mathcal{N}} \beta_U(u) A_i(a_U), \quad \forall u \in \widetilde{D}_i.
\]
Since the partition of unity $\beta_U(u)$ is locally Lipschitz continuous and the operator $A_i$ is continuous, $B(u)$ is a locally finite convex combination. Thus, $B : \widetilde{D}_i \to D_{i,+}^{dd}$ is also locally Lipschitz. By the definition of $B(u)$ and Lemma 3.4, when $u \in  D_i^\pm(\mu)$, we have $A_i(a_U) \in  D_i^\pm(\frac{1}{2}\mu)$. Since $  D_i^\pm(\frac{1}{2}\mu)$ is convex, $B(u)$, being a convex combination of the $A_i(a_U)$, satisfies
\[
B(u) \in  D_i^\pm(\frac{1}{2}\mu).
\]
It follows that
\[
B\left(\ov{D}_1^{D_{i,+}^{dd}}\setminus K\right)  \subset D_i^+(\frac{1}{2}\mu), \quad B\left(\ov{D}_2^{D_{i,+}^{dd}}\setminus K\right)  \subset D_i^-(\frac{1}{2}\mu).
\]

Now, for any $u \in \widetilde{D}_i$, by the local finiteness of $\mathcal{N}$, only finitely many sets in $\mathcal{N}$ contain $u$. We denote these sets by $N_i$ ($i = 1, 2, \dots, n(u)$), so that $u \in N_i$. For each $i$, there exists $M_i \in \mathcal{M}$ such that $N_i \subset M_i$, and hence $a_{N_i} \in N_i \subset M_i$ and $u \in N_i \subset M_i$. From this analysis, it follows that for every $u \in \widetilde{D}_i$, satisfies
\begin{align*}
\|B(u) - A_i(u)\| &= \left\| \sum_{U \in \mathcal{N}} \beta_U(u)A_i(a_U) - \sum_{U \in \mathcal{N}} \beta_U(u)A_i(u) \right\| \\
&\ls \sum_{U \in \mathcal{N}} \beta_U(u)\|A_i(a_U) - A_i(u)\|.
\end{align*}
From (3.7), (3.8), (3.9), for all $u \in \widetilde{D}_i$, the following  inequality holds:
\[
\|B(u) - A_i(u)\| \ls \frac{1}{2}\|u - A_i(u)\|.
\]
Therefore,
\[
\|u - B(u)\| \ls \|u - A_i(u)\| + \|A_i(u) - B(u)\| \ls \frac{3}{2}\|u - A_i(u)\|,
\]
\[
\|u - A_i(u)\| \ls \|u - B(u)\| + \|B(u) - A_i(u)\| \ls \|u - B(u)\| + \frac{1}{2}\|u - A_i(u)\|.
\]
In conclusion, for any $u \in \widetilde{D}_i$, we have
\[
\frac{1}{2}\|u - A_i(u)\| \ls \|u - B(u)\| \ls 2\|u - A_i(u)\|.
\]
On the other hand,
\begin{align*}
( J'(u), u - B(u) ) &=  \Big( J'(u), u - \sum_{i=1}^{n(u)} \beta_{N_i}(u)A_i(a_{N_i}) \Big) \\
&= \sum_{i=1}^{n(u)} \beta_{N_i}(u)( J'(u), u - A_i(a_{N_i}) ) \\
&= \sum_{i=1}^{n(u)} \beta_{N_i}(u)\left[ (J'(u) , u - A_i(u) ) + ( J'(u), A_i(u) - A_i(a_{N_i}) ) \right] \\
&\gl \sum_{i=1}^{n(u)} \beta_{N_i}(u)\left[\| u - A_i(u)\|^2 - \|J'(u)\|\cdot\|A_i(u) - A_i(a_{N_i})\| \right].  \\
\end{align*}
From (3.8), we have
\[
\|J'(u)\| \cdot \|A_i(u) - A_i(a_{N_i})\| \ls\frac{1}{2}  \|u - A_i(u)\|^2.
\]
Thus, (3.7) holds. In conclusion, we have proved that operator \( B \) satisfies properties (1)-(3) in the lemma, and the proof is complete.
\end{proof}
\medskip

 Next, we consider the following initial value problem in Banach space.
\[\left\{\be{ll}
\di\frac{d\si(t)}{dt}=B(\si(t)) - \si(t)\\
\si(0)=x_{0}\in\widetilde{D}_i.
\en\right.
\no(3.10)\]
Since \(W:=I - B\) is locally Lipschitz continuous, according to the theory of ordinary differential equations in Banach space, \((3.10)\) has a right saturated solution \(\si^t(x_{0})\in D_{i,+}^{dd}\) for \(t\in[0, T(x_{0}))\), where \(0<T(x_{0})\ls+\infty\).
For any given \(t\in[0,T(x_{0}))\), we have
\[\be{ll}\di\frac{dJ(\si^t(x_0))}{dt}&=\Big( J^{\prime}(\si^t(x_0)),\di\frac{d\si^t(x_0)}{dt}\Big)\\
&=-( J^{\prime}(\si^t(x_0)),\si^t(x_0) -B(\si^t(x_0)))\\
&\ls-\frac{1}{2}\|\si^t(x_0) - A_i(\si^t(x_0))\|^{2}<0.
\en\no(3.11)\]
This shows that \(J(\si^t(x_0))\) is decreasing with respect to \(t\in[0,T(x_{0}))\). Hence, as in \cite{s19} we have the following Definition 3.2 and 3.3.

\begin{definition}\
Let \(D\subset D_{i,+}^{dd}\). If for any given \(x_{0}\in D\backslash K\), we have
\[\{\si^t(x_0):t\in[0,T(x_{0}))\}\subset D,\]
then \(D\) is called a  invariant set of descending flow of (3.10).
\end{definition}

\begin{definition}
Let \( M \) and \( D \) be  invariant sets of the descending flow of \((3.10)\), with \( D \subset M \). Define
\[
C_M(D) = \left\{ x_0: x_0 \in D, \text{ or } x_0 \in M \setminus D \text{ and there exists } t' \in [0, T(x_0)) \text{ such that } \si^{t'}( x_0) \in D \right\}.
\]
If \( C_M(D) = D \), then \( D \) is called a complete invariant set of the descending flow relative to \( M \).
\end{definition}
By Theorem 2.1 and 2.2 in \cite{s19}, we have the following  Lemma 3.6 and 3.7.
\begin{lemma}\  Let $D$ be a closed invariant set of the descending flow of  (3.10). Assume that $J$ satisfies the (P.S.) condition on $D$ and $c:=\inf_{u\in D}J(u)>-\infty$. Then $c$ is a critical value of $J$. Moreover, there exists $x_0\in D\cap K$ such that $J(x_0)=c$.
\end{lemma}

\begin{lemma}\ Let $G\subset E$  be a connected and invariant set of (3.10), and $D$ be an open invariant subset of $G$. Then the following assertions hold:
\\
  1)\ $C_G(D)$ is  an open subset of $G$;\\
2)\ $\pa_G C_G(D)$ is  an invariant set of descending flow of (3.10);\\
 3)\ $\inf\limits_{u\in \pa_G C_G(D)} J(u)\gl \inf\limits_{u\in \pa_G D} J(u)$.
\end{lemma}

Let $\mu_0$ be given as in Lemma 3.4 and $\mu\in (0,\mu_0]$.  For any \( c \in \mathbb{R} \), define:
\[
Q^c = (J^c \cap D_{i,+}^{dd}) \cup \overline{D}_1^{D_{i,+}^{dd}} \cup \overline{D}_2^{D_{i,+}^{dd}}, \quad K_c^* = (K_c \cap D_{i,+}^{dd}) \setminus (\overline{D}_1^{D_{i,+}^{dd}} \cup \overline{D}_2^{D_{i,+}^{dd}}),
\]
\[
Q_c^* = Q^c \setminus K_c^*.
\]
For \( c, d \in \mathbb{R} \) with \( c < d \), define:
\[
M_c^d = Q_d^* \setminus Q^c = (J^d \cap D_{i,+}^{dd}) \setminus \left( J^c \cup K_d \cup \overline{D}_1^{D_{i,+}^{dd}} \cup \overline{D}_2^{D_{i,+}^{dd}} \right).
\]

Assume that $M_c^d\cap K=\ep$. For any \( x_0 \in M_c^d \), consider the following initial value problem:
\[
\begin{cases}
\frac{d \psi(t)}{dt} = -\frac{W(\psi(t))}{\| J'(\psi(t)) \|^2}, \\
\psi(0) = x_0\in\widetilde{D}_i.
\end{cases} \tag{3.12}
\]
Denote the  right saturated solution  as \(\psi^t(x_{0})\in D_{i,+}^{dd}\) for \(t\in[0, T_1(x_{0}))\), where \(0<T_1(x_{0})\ls+\infty\). Let
\[
\psi^s(x_0) =  \sigma^t( x_0), \quad s = \int_0^t \| J'(\sigma^\tau(x_0)) \|^2 d\tau.
\]
Then \( \psi^s(x_0) \) is a reparametrization of \( \si^t(x_0) \). In fact, for $s\in [0, T_1(x_0))$ we have
\[
\frac{\partial \psi^s(x_0)}{\partial s} = \frac{\partial \sigma^t(x_0)}{\partial t} \bigg/ \frac{ds}{dt} = -\frac{W(\sigma^t(x_0))}{\| J'(\sigma^t(x_0)) \|^2} = -\frac{W(\psi^s(x_0))}{\| J'(\psi^s(x_0)) \|^2}.
\]
By the uniqueness of the solution to (3.12), \( \psi^s(x_0) \) is the solution to (3.12).

Next, we define
\[
\tau(x_0) = \sup\left\{ \tau: 0 < \tau < T(x_0),\ \si^t( x_0) \in M_c^d \text{ for } 0 \ls t \ls \tau \right\},
\]
\[
\eta(x_0) = \int_0^{\tau(x_0)} \| J'(\sigma^t(x_0)) \|^2 dt,
\]
i.e., \( \tau(x_0) \) is the maximal time that \( \sigma^t(x_0) \) stays in \( M_c^d \), and \( \eta(x_0) \) is its reparametrization.

\begin{lemma}\cite{s21}
Let $C$ be a convex set of $X$. Then for all $x\in\mbox{int}\ C$ and $y\in C$, $\al x+(1-\al)y\in \mbox{int}\ C$  for all $\al\in (0,1]$.
\end{lemma}

\begin{lemma} Assume that  (H\(_{1})\),  (H\(_{4})\) and (H\(_{5})\) hold.
$\overline{D}_1^{D_{i,+}^{dd}}$,   $\overline{D}_2^{D_{i,+}^{dd}}$ and $\overline{D}_1^{D_{i,+}^{dd}}\cap \overline{D}_2^{D_{i,+}^{dd}}$ are invariant sets of (3.10) and (3.12).
\end{lemma}

\begin{proof}
For $u\in \partial_{D_{i,+}^{dd}} D_1$, it follows from Lemma 3.5 and 3.8 that for $\la>0$ small enough,
$$ u+\la (-W(u))
=\la B(u)+(1-\la) u\in  D_1.
$$
  It follows from the theorem due to Brezis-Martin (see \cite{s22}) that $\overline{D}_1^{D_{i,+}^{dd}}$ is an invariant sets of descending flow of (3.10) and (3.12).
  Similarly, $\overline{D}_2^{D_{i,+}^{dd}}$ and $\overline{D}_1^{D_{i,+}^{dd}}\cap \overline{D}_2^{D_{i,+}^{dd}}$  are invariant sets of (3.10) and (3.12).
 \end{proof}

\begin{lemma} Assume that   (H\(_{1})\),(H\(_{4})\) and (H\(_{5})\) hold,  $M_c^d\cap K=\ep$.  Then
\( \eta: M_c^d \to \mathbb{R}_+ \) is continuous.
\end{lemma}

\begin{proof}
Take any \( x_0 \in M_c^d \), and let \[\zeta(x_0):= \lim_{t \to \tau^-(x_0)} \sigma^t(x_0) = \lim_{s \to \eta^-(x_0)} \psi^s(x_0).\] We first prove that if \( \zeta(x_0) \in J^{-1}(c) \), then \( \eta(x_0) < +\infty \). In fact, for any \( 0 \ls s < \eta(x_0) \), we have
\[
\frac{d J(\psi^s(x_0))}{ds} = -\frac{( J'(\psi^s(x_0)), W(\psi^s(x_0)) )}{\| J'(\psi^s(x_0)) \|^2} \ls -\frac{1}{2}.
\]
Thus,
\[
-\int_0^{\eta(x_0)} \frac{d}{ds} J(\psi^s(x_0)) ds = J(\psi^0(x_0)) - \lim_{s \to \eta^-(x_0)} J(\psi^s(x_0)) \ls d - c.
\]
Hence,
\[
\eta(x_0) \ls 2(d - c) < +\infty.
\]

Since $x\in K$ if and only if $x=A_i(x)$,  it follows from Lemma 3.4  that $K\cap \partial _{D_{i,+}^{dd}}(D_1 \cup D_2)=\ep$. For \(\zeta(x_0) \), we have the following four cases:
\begin{enumerate}
    \item \(\zeta(x_0) \in (J^d \setminus J^c) \cap \partial_{D_{i,+}^{dd}}(D_1 \cup D_2) \),
    \item \(\zeta(x_0) \in (J^{-1}(c) \cap {D_{i,+}^{dd}}) \setminus (D_1 \cup D_2 \cup K) \),
    \item \(\zeta(x_0) \in J^{-1}(c)  \cap \partial _{D_{i,+}^{dd}}(D_1 \cup D_2) \),
    \item \(\zeta(x_0) \in K_c^*\bk(D_1\cup D_2) \).
\end{enumerate}

First, consider case (1). Assume without loss of generality that $\zeta(x_0)=\psi^{\eta(x_0)}(x_0)\in \partial _{D_{i,+}^{dd}}D_1$. By Lemma 3.9 we have for any $\va>0$,
$$\{\psi^s(x_0): s\in [\eta(x_0), \eta(x_0)+\va]\}\subset \overline{D}_1^{D_{i,+}^{dd}}.$$
Consequently,  we have by Lemma 3.4, $$B(\{\psi^s(x_0): s\in [\eta(x_0), \eta(x_0)+\va]\})\subset D^+_1(\frac{1}{2}\mu).$$ Since $\mbox{Cl}_{D_{i,+}^{dd}}D^+_1(\frac{1}{2}\mu)$ is a convex set,  we have
\[
\frac{1}{e^t - 1}\int_1^{e^t} B(\psi^s(\eta(x_0))) ds \in \ov{D_i^+(\frac{1}{2}\mu)}^{D_{i,+}^{dd}}\subset  \mbox{Cl}_{D_{i,+}^{dd}}D^+_1(\frac{1}{2}\mu).
\]
It follows from Lemma 3.8 that for $t>0$,
\[
\psi^t(\zeta(x_0)) = e^{-t}\zeta(x_0) + (1-e^{-t}) \cdot \frac{1}{e^t - 1}\int_1^{e^t} B(\psi^s(\zeta(x_0))) ds \in  D_1.
\]
In particular, we have \( \psi^{\eta(x_0)+\va}(x_0) \in (D_1 \cup D_2) \).
By the definition of \( \eta(x_0) \), we have for any $\va>0$,
\[
\left\{ \psi^s(x_0): 0 \ls s \ls \eta(x_0) - \varepsilon \right\} \subset M_c^d.
\]
 By the continuous dependence of solutions to ODEs in Banach spaces on initial values, there exists \( \delta_1 > 0 \) such that for \( \| x - x_0 \| < \delta_1 \),
\[
\left\{ \psi^s(x): 0 \ls s \ls \eta(x_0) - \varepsilon \right\} \subset M_c^d.
\]
This implies that for \( \| x - x_0 \| < \delta_1 \), \( \eta(x) > \eta(x_0) - \varepsilon \). Also, since \( \psi^{\eta(x_0)+\va}(x_0) \in D_1 \cup D_2 \), by the continuous dependence of solutions on initial values, there exists \( \delta_2 > 0 \) such that for \( \| x - x_0 \| < \delta_2 \), \( \psi^{\eta(x_0)+\va}(x) \in D_1 \cup D_2 \). Let \( \delta_0 = \min\{\delta_1, \delta_2\} \). Then for \( \| x - x_0 \| < \delta_0 \), \( \eta(x) > \eta(x_0) - \varepsilon \) and \( \eta(x) < \eta(x_0) + \varepsilon \), so
\[
|\eta(x) - \eta(x_0)| < \varepsilon.
\]
Thus, \( \eta: M_c^d \to \mathbb{R}^+ \) is continuous in case (1).

Next, consider case (2). We consider the equation
\[
J(\psi^s(x)) =c . \tag{3.13}
\]
Obviously, \( J(\psi^{\eta(x_0)}(x_0))=c \). Also,
\[
\frac{\partial (J(\psi^s(x)) - c)}{\partial s} = -\frac{( J'(\psi^s(x)), W(\psi^s(x)) )}{\| J'(\psi^s(x)) \|^2} < -\frac{1}{2} < 0.
\]
By the implicit function theorem, there exists a neighborhood \( U \) of \( x_0 \) and a continuous function \( s: U \to \mathbb{R} \) such that \( s(x_0) = \eta(x_0) \) and
\[
J(\psi^{s(x)}(x)) =c.
\]
The uniqueness of $s(x)$ in $U$ implies that $s(x)=\eta(x)$ for $x\in U$. So, \( \eta: M_c^d \to \mathbb{R}_+ \) is continuous in case (2).

Case (3) satisfies both case (1) and case (2), so we can use the methods from (1) and (2) to show that \( \eta: M_c^d \to \mathbb{R}_+ \) is continuous in case (3).

Finally, we consider case (4).
If \( s < \eta(x_0)- \varepsilon \), then \( J(\psi^s(x_0)) > c \). By continuity, there exists a neighborhood \( U \) of \( x_0 \) such that for any \( x \in U \), \( J(\psi^s(x)) > c \) for \( s < \eta(x_0)- \varepsilon \), so \( \eta(x) \gl \eta(x_0) - \varepsilon \).

To prove continuity, we use contradiction. Suppose \( \eta \) is not continuous at \( x_0 \). Then there exists \( \varepsilon_0 > 0 \) and a sequence \( \{x_n\} \subset M_c^d \) with \( x_n \to x_0 \) as \( n \to \infty \) such that \( \eta(x_n) > \eta(x_0) + \varepsilon_0 \). Then, we have for $\va\in(0,\va_0)$,
\[
\begin{aligned}
J\left(\psi^{\eta(x_0) - \varepsilon}(x_n)\right) - J\left(\psi^{\eta(x_0) + \varepsilon_0}(x_n)\right)
&= \int^{\eta(x_0) + \varepsilon_0}_{\eta(x_0) - \varepsilon}- \frac{d}{ds} J\left(\psi^s(x_n)\right) ds \\
&= \int^{\eta(x_0) + \varepsilon_0}_{\eta(x_0) - \varepsilon} \frac{( J'(\psi^s(x_n)), W(\psi^s(x_n)) )}{\| J'(\psi^s(x_n)) \|^2} ds \\
&\gl \frac{1}{2} (\varepsilon_0 +\varepsilon).
\end{aligned}
\]
Since \( \eta(x_n) > \eta(x_0) + \varepsilon_0 \), we have
\[
J\left(\psi^{\eta(x_0) - \varepsilon}(x_n)\right) \gl c + \frac{1}{2} (\varepsilon_0+\varepsilon).
\]
Letting \( x_n \to x_0 \) and \( \varepsilon \to 0 \), we get \(J\left(\psi^{\eta(x_0)}(x_0)\right) \gl c + \frac{1}{2}\varepsilon_0 \), which contradicts \( J\left(\psi^{\eta(x_0)}(x_0)\right) = c \).
\end{proof}

\begin{remark}By a similar way as the proof of the Case (1) in the above lemma we can show that $D_1$ and $D_2$ are invariant sets of (3.10). It can be observed from the proof of Case (1) above that the flow of (3.10) is non-trapping on the boundaries of $D_1$ and $D_2$. The conditions \(B\left(\mbox{Cl}_{D_{+,i}^{dd}}D_i^{\pm}(\mu)\setminus K\right) \subset D_i^{\pm}(\frac{1}{2}\mu)\)  play a crucial role in guaranteeing this property.
\end{remark}

\begin{lemma} Assume that   (H\(_{1})\), (H\(_{4})\) and  (H\(_{5})\) hold, $M_c^d\cap K=\ep$ and $K_c\bk (P\cup (-P))$ is a finite number set.  Then \( Q^c \) is a strong deformation retract of \( Q_d^* \).
\end{lemma}

\begin{proof}
Define \( \varphi: [0,1] \times Q_d^* \to Q_d^* \) as
\[
\varphi(t,x) =
\begin{cases}
x, & \text{if } (t,x) \in [0,1] \times (Q^c \cup \ov D_1^{D_{i,+}^{dd}} \cup \ov D_2^{D_{i,+}^{dd}}), \\
\psi^{t \cdot \eta(x)}(x), & \text{if } (t,x) \in [0,1) \times M_c^d, \\
\zeta(x), & \text{if } (t,x) \in \{1\} \times M_c^d.
\end{cases}
\]

To prove that \( \varphi \) is continuous at \( (t_0,x_0) \in [0,1] \times Q_d^* \), according to the position of $x_0$ we need to consider several cases. For brevity we only consider  the following two cases. The other cases are similar.

1.\ If \( x_0 \in \partial_{D_{i,+}^{dd}}(D_1 \cup D_2)\bk J^c \), then \( \eta(x_0) = 0 \), so \( \varphi(t,x_0) = x_0 \) for all \( t \in [0,1] \). Take $\de_1>0$ such that $B(x_0, \de_1)\cap J^c=\ep$. Obviously, there exists $s(t)\in [0,1]$ such that
    \[
   t \eta(x) = \int_0^{s(t)\tau(x)} \| J'(\si^t(x)) \|^{2} dt,
    \]
    where \( \si^t(x) \) satisfies
    \[
    \frac{d \si^t(x)}{dt} = -W(\si^t(x)),\quad \si^0(x) = x.
    \]
    For any \( t \in [0,1] \) and $x\in M_c^d$, we have
    \[
    \begin{aligned}
    \|\varphi(t,x) - \varphi(t,x_0)\| &= \|\psi^{t \cdot \eta(x)}(x) - x_0\| \\
    &\ls \|x - x_0\|+\|\si^{s(t) \tau (x)}(x) - x\|  \\
    &\ls \|x - x_0\| + \int_0^{s(t) \cdot \tau(x)} \|W(\si^s(x))\| ds \\
    &\ls \|x - x_0\| + 2 \left( \int_0^{s(t) \cdot \tau(x)} \|J'(\si^s(x))\|^2 ds \right)^{1/2} \cdot \sqrt{s(t) \cdot \tau(x)}.
    \end{aligned}
    \tag{3.14}
    \]
    On the other hand,
    \[
    \frac{d J(\si^t(x))}{dt} = -( J'(\si^t(x)), W(\si^t(x)) ) \ls -\frac{1}{2} \|J'(\si^t(x))\|^2,
    \]
    so
    \[\be{ll}
    \di\int_0^{s(t) \cdot \tau(x)} \|J'(\si^s(x))\|^2 ds &\ls 2 \di\int_0^{\tau(x)} -\frac{d J(\si^s(x))}{ds} ds\\
    & = 2\left[ J(x) - J(\si^{\tau(x)}(x)) \right] \ls 2(d - c).\en
    \tag{3.15}
    \]
    From (3.14) and (3.15), for any \( t \in [0,1] \), we have
\[
\|\varphi(t,x) - \varphi(t_0,x_0)\| \ls \|x - x_0\| + 2\sqrt{2}(d - c)^{\frac{1}{2}} \sqrt{\tau(x)} \tag{3.16}
\]
Since \( \tau(x_0) = 0 \) and \( \tau \) is continuous, for any \( \varepsilon > 0 \), there exists \( 0<\delta < \min\left\{\frac{\varepsilon}{2}, \delta_1\right\} \), such that for any \( x \in B(x_0,\delta) \), we have \( 0\ls\tau(x) < \frac{\varepsilon^2}{8\sqrt{(d - c)}} \). Then, by (3.16), for any \( x \in B(x_0,\delta) \cap M_c^d \), we have
\[
\|\varphi(t,x) - \varphi(t_0,x_0)\| < \varepsilon,
\]
and for any \( x \in B(x_0,\delta) \cap (D_1 \cup D_2) \) and \( t \in [0,1] \),
\[
\|\varphi(t,x) - \varphi(t_0,x_0)\| = \|x - x_0\| < \frac{\varepsilon}{2} < \varepsilon.
\]
Thus, \( \varphi(t,x) \) is continuous at \( (t_0,x_0) \).

\medskip

2. \( x_0 \in K_c \).  Clearly,  $\varphi(t_0, x_0)=x_0$ for each $(t_0, x_0)\in [0,1]\ti K_c$. Suppose \( \varphi \) is not continuous at \( (t_0,x_0) \). Then there exists \( \varepsilon_0 > 0 \), \( \{t_n\} \subset [0,1] \), and \( \{x_n\} \subset M_c^d \) such that \( t_n \to t_0 \), \( x_n \to x_0 \), and
\[
\|\varphi(t_n,x_n) - \varphi(t_0,x_0)\| = \|\varphi(t_n,x_n) - x_0\| \gl \varepsilon_0.
\]
Let $t_0\in (0,1)$. Suppose without loss of generality that  and $t_n\gl t_0$. Since \( K_c \) contains finitely many critical points, we can assume \( x_0 \) is an isolated critical point of \( J \). By the (P.S.) condition, there exists \( \va_0>\delta_0 > 0 \) such that for \( x \in M_c^d \cap \left( \overline{B}(x_0,\delta_0) \setminus B(x_0,\frac{\delta_0}{2}) \right) \), \( \|J'(x)\| \gl \delta_0 \).  For each \( j \in \mathbb{N} \)  large enough, take \( n_j\in \mathbb N\) such that $t_{n_j}\ls t_0+\frac{1}{j}$  and \( x_{n_j} \) such that \( \left\| \varphi\left(t_0 + \frac{1}{j},x_{n_j}\right) - \varphi\left(t_0 + \frac{1}{j},x_0\right) \right\| < \frac{1}{j} \), so
\[
\lim_{j \to \infty} \left\| \varphi\left(t_0+ \frac{1}{j},x_{n_j}\right) - x_0 \right\| = 0.
\]
For sufficiently large \( j \), we can take \( t_j' \) and \( t_j'' \) such that \( t_0+ \frac{1}{j} >t_j'> t_j'' \gl t_{n_j} \) and
\[
\|\varphi(t_j',x_{n_j}) - x_0\| = \frac{\delta_0}{2}, \quad \|\varphi(t_j'',x_{n_j}) - x_0\| = \delta_0,
\]
and
\[
\|J'(\varphi(t,x_{n_j}))\| \gl \delta_0, \quad \forall t \in [t_j',t_j''].
\]
Then, we have
\[\be{ll}
\frac{\delta_0}{2} &\ls d\left(\varphi(t_j',x_{n_j}), \varphi(t_j'',x_{n_j})\right)\\
& \ls\di \int^{t_j'}_{t_j''} \left\| \frac{d}{dt}\varphi(t,x_{n_j}) \right\| dt\\
&\ls\di \int^{t_j'}_{t_j''} \frac{dt}{\|J'(\varphi(t,x_{n_j}))\|}\\
& \ls \frac{1}{\delta_0}(t_j'-t_j'' )\\
& \ls \frac{1}{\delta_0}\left( t_0 + \frac{1}{j}-t_{n_j}\right).\en
\]
Letting \( j \to \infty \), we get a contradiction. In the same way we can prove the cases $t_0=0$ and $t_0=1$.
Thus, \( \varphi \) is continuous at \( (t_0,x_0) \). This completes the proof.\end{proof}

\begin{lemma}[Nontrivial interval theorem \cite{s28}] If $\exists q\in\mb N$ and $\exists c<d$ such that $H_q(J^d, J_c)$ is nontrivial, then $K\cap J^{-1}[c,d]\neq\ep$.\end{lemma}

\medskip

{\bf \large The Proof of Theorem 3.2}

\begin{proof}
 Let $i\in\Lambda_1$ be fixed and let $V_1=\ov{D}_1^{D_{i,+}^{dd}}\cap D_2$. Then $V_1$ is a nonempty open invariant set of (3.10) in $\ov{D}_1^{D_{i,+}^{dd}}$. We claim that $J$ is bounded from below on $V_1$. In fact, for each $u\in V_1$, we have $d(u, P)\ls \mu$ and $d(u,-P)\ls \mu$. Since $X$ is a real Hilbert space, $P$ and $-P$ are closed convex sets in $X$, then $d(\cdot, P)$ and $d(\cdot, -P)$ are attainable. Hence,  for each $u\in V_1$, we may take $w_1\in P$ and $w_2\in -P$ such that $d(u, P)=\|u-w_1\|$ and $d(u, -P)=\|u-w_2\|$. Write $\xi_u=u-w_1$ and $\eta_u=u-w_2$. Then we have $w_1-w_2=\eta_u-\xi_u$. Since $X\hookrightarrow Z$, we have
 $$\|\xi_u\|_2\ls c_1\|\xi_u\|\ls c_1\mu,\  \|\eta_u\|_2\ls c_1\|\eta_u\|\ls c_1\mu.$$
Note that $Z$ is a Banach lattice and $0\ls w_1\ls w_1-w_2$. Hence, we have
$$\|w_1\|_2\ls \|w_1-w_2\|_2=\|\eta_u-\xi_u\|_2\ls \|\eta_u\|_2+\|\xi_u\|_2\ls 2 c_1\mu,$$
and so $$\|u\|_2\ls \|w_1\|_2+\|\xi_u\|_2\ls 3c_1\mu.$$
By (H$_5)$ we have  for all $u\in V_1$,
$$|F(u)|=\left|\int_0^1\langle \mathbf{f}(su), u\rangle ds\right|\ls g(s\|u\|_2)\|u\|_2\ls 3c_1\mu g(3 c_1\mu), $$ and so $J(u)\gl -3c_1\mu g(3 c_1\mu)$. Hence,  $J$ is bounded from below on $V_1$.

It follows from (H$_2)$ that for each $u\in (P\bk\{0\})
 \cap \ov{D}_1^{D_{i,+}^{dd}}$,  $J(tu)\ri-\infty$ as $t\ri+\infty$. Thus, $C_{\ov{D}_1^{D_{i,+}^{dd}}} V_1\neq \ov{D}_1^{D_{i,+}^{dd}}$, and so $\pa_{\ov{D}_1^{D_{i,+}^{dd}}} C_{\ov{D}_1^{D_{i,+}^{dd}}} V_1\neq\ep$. According Lemma 3.7, we have
\[c_{i,+}:=\inf_{u\in \pa_{\ov{D}_1^{D_{i,+}^{dd}}} C_{\ov{D}_1^{D_{i,+}^{dd}}} V_1} J(u)\gl \inf_{u\in \pa_{\ov{D}_1^{D_{i,+}^{dd}}} V_1}J(u)>-3c_1\mu g(3 c_1\mu).\]
By Lemma 3.6, $c_{i,+}$ is a critical value of $J|_{D_{i,+}^{dd}}$, and there exists $x_{i,+}\in \pa_{\ov{D}_1^{D_{i,+}^{dd}}} C_{\ov{D}_1^{D_{i,+}^{dd}}} V_1$ such that
\[(x_{i,+}, w)=\langle \mathbf{f}(x_{i,+}), w\rangle, \ \forall w\in D_{i,+}^{dd}.\]
An argument similar to Theorem 3.1 yields that
\[(x_{i,+}, w)=\langle \mathbf{f}(x_{i,+}), w\rangle, \ \forall w\in X.\]
This implies that $x_{i,+}\in K$. Since $A_i(D_1)\subset D^+_i(\frac{1}{2}\mu)$, we have $K\cap (\ov{D}_1^{D_{i,+}^{dd}}\bk P)=\ep$. So, $x_{i,+}\in K\cap (P\bk\{0\})$. Hence, $J$ has at least $2^{m_1}-1$ positive critical points such that each of these positive critical points has the form $\bigsqcup\limits_{i=1}^{m_1}(b_i x_{i,+})$, where $b_i\in \{0,1\}$ for $i\in\Lambda_1$  with $0\neq (b_1, b_2,\cdots, b_{m_1})\in \mb R^{m_1}$.

Let $V_2=\ov{D}_2^{D_{i,+}^{dd}}\cap D_1$. By a similar way as above, we can show that there exists $y_{i,+}\in K\cap (-P) $ such that $y_{i,+}\in \pa_{\ov{D}_2^{D_{i,+}^{dd}}} C_{\ov{D}_2^{D_{i,+}^{dd}}}V_2$. Hence, $J$ has at least $2^{m_1}-1$ negative critical points such that each of these negative critical points has the form $\bigsqcup\limits_{i=1}^{m_1}(b_i y_{i,+})$, where $b_i\in \{0,1\}$ for $i\in\Lambda_1$  with $0\neq (b_1, b_2,\cdots, b_{m_1})\in \mb R^{m_1}$.

By (H$_2)$, we may take  paths $\ga_1$ and $\ga_2$ such that \[ \gamma_1 \subset \ov{D}^{D_{i,+}^{dd}}_1\cup \ov{D}^{D_{i,+}^{dd}}_2,   \gamma_2 \subset D_{i,+}^{dd}\bk\big(\ov{D}^{D_{i,+}^{dd}}_1\cup \ov{D}^{D_{i,+}^{dd}}_2\big), \]\[ \gamma_1(0) = \gamma_2(0) = \overline{x}_1\in D_1\bk \ov D_2^{D_{i,+}^{dd}}, \gamma_1(1) = \gamma_2(1) = \overline{x}_2\in D_2\bk \ov D_1^{D_{i,+}^{dd}},\]and \[
\max_{t \in [0,1]} J(\gamma_2(t)) < \max_{t \in [0,1]} J(\gamma_1(t)).
\] Moreover,  we may assume \( \gamma_1 \) and \( \gamma_2 \) each has no loops.

Define the map \( i^c: \gamma_1 \cup \gamma_2 \to Q^c := (J^c\cap D_{i,+}^{dd}) \cup \ov{D}^{D_{i,+}^{dd}}_1 \cup \ov{D}^{D_{i,+}^{dd}}_2 \).
Let \( i_*^c: H_1(\gamma_1\cup  \gamma_2) \to H_1(Q^c) \) be the induced homomorphism. Define
\[
c_1: = \sup\left\{ c > c_0: i_*^c \text{ is monomorphism} \right\},
\]
where \( c_0 = \inf_{h \in \Gamma} \max_{t \in [0,1]} J(h(t)) \), and \[ \Gamma = \left\{ h: [0,1] \to \ov{D}^{D_{i,+}^{dd}}_1 \cup \ov{D}^{D_{i,+}^{dd}}_2:  h \text{ is continuous}, h(0)=\overline{x}_1, h(1)=\overline{x}_2 \right\}.\]

Let \( P = \left\{ s\gamma_1(t) + (1-s)\gamma_2(t) \mid s,t \in [0,1] \right\} \). Clearly, for \( c'\gl \max_{p \in P} J(p) \), we have \( i_*^{c'} = 0 \) and so \( c_1 \ls \max_{p \in P} J(p) \). Next, we show \( c_1 \gl c_0 \). Let $c'\in\mb R$ be such that  \( \max_{t\in [0,1]}J(\ga_2(t))<c' < c_0 \). Set \[ G_1 = J^{c'}\cap D_{i,+}^{dd}, \  G_2 = \ov {D}^{D_{i,+}^{dd}}_1 \cup \ov{D}^{D_{i,+}^{dd}}_2. \] Consider the following Mayer-Vietoris sequence:
\[
\xymatrix{
H_1(\ga_1) \oplus H_1(\ga_2) \ar[r] & H_1(\ga_1 \cup \ga_2) \ar[r]^{\partial} \ar[d]^{i_*^{c'}} & \tilde{H}_0(\ga_1 \cap \ga_2) \ar[r] \ar[d]^{j_*} & \tilde{H}_0(\ga_1) \oplus \tilde{H}_0(\ga_2) \\
H_1(G_1) \oplus H_1(G_2)\ar[r] & H_1(G_1 \cup G_2) \ar[r]^{\partial} & \tilde{H}_0(G_1 \cap G_2) \ar[r] & \tilde{H}_0(G_1) \oplus \tilde{H}_0(G_2)
}
\]
where the maps $j_*$ is induced by inclusion. Since $$H_1(\ga_1)\oplus H_1(\ga_2)=\widetilde H_0(\ga_1)\oplus\widetilde H_0(\ga_2)=0,$$ the exactness in the first row implies that $\pa$ is an isomorphism. Since $c'< c_0$ there is no path in  \( \big(J^{c'}\cap D_{i,+}^{dd}\big ) \cap \big(\ov {D}^{D_{i,+}^{dd}}_1 \cup \ov{D}^{D_{i,+}^{dd}}_2\big) =G_1\cap G_2\) joining $\bar{x}_1$ and $\bar x_2$.  Thus $\ga_1\cap \ga_2$  consists of two points which are in different path components of $G_1\cap G_2$. Hence $j_*$ is a monomorphism which implies  that $i_*^{c'}$ is also a monomorphism. So,  \( c_1 \gl c_0 \).

For any \( \varepsilon > 0 \) small, we have the following commutative diagram:
\[
\xymatrix{
H_2(Q^{c_1+\va}, Q^{c_1-\va}) \ar[d]^{\pa} \\
H_1(Q^{c_1-\va}) \ar[d]^{i_*} & H_1(\ga_1\cup \ga_2) \ar[l]^{i_*^{c_1-\va}} \ar[ld]^{i_*^{c_1+\va}} \\
H_1(Q^{c_1+\va})
}
\]
If \( H_2(Q^{c_1+\varepsilon}, Q^{c_1-\varepsilon}) = 0 \), then \( i_* \) is a monomorphism.  By the definition, \( i_*^{c_1-\va} \) also is a monomorphism. Thus, \( i_* \circ i_*^{c_1-\va} \) is a monomorphism, which contradicts that $i_*^{c_1+\va}=i_* \circ i_*^{c_1-\va}$ is a zero-morphism.
Hence, \( H_2(Q^{c_1+\varepsilon}, Q^{c_1-\varepsilon}) \neq 0 \). By the excision property of homology, \[ H_2\Big((J^{c_1+\varepsilon}\cap D_{i,+}^{dd}) \bk \big(\ov{D}^{D_{i,+}^{dd}}_1 \cup\ov{D}^{D_{i,+}^{dd}}_2\big), (J^{c_1-\varepsilon}\cap D_{i,+}^{dd}) \bk \big(\ov{D}^{D_{i,+}^{dd}}_1 \cup\ov{D}^{D_{i,+}^{dd}}_2\big)\Big)\cong  H_2(Q^{c_1+\varepsilon}, Q^{c_1-\varepsilon})\neq 0. \] It follows from Lemma 3.12 that there at least exists a critical point in \( M^{c_1+\varepsilon}_{c_1+\varepsilon}\). Since $\va>0$ is arbitrarily given and $J$ satisfies the (P.S.) condition, we see that $c_1$ is a critical value of $J$,  and  there exists $z_{i,+}$ such that \[z_{i,+}\in \big(J^{-1}(c_1)\cap K\cap D_{i,+}^{dd}\big)\bk   \big(\ov{D}^{D_{i,+}^{dd}}_1 \cup\ov{D}^{D_{i,+}^{dd}}_2\big).\]
An argument similar to Theorem 3.1 yields that $J$ has at least $2^{2m_1}-1$ nontrivial critical points such that each of these  critical points has the form \[\left(\bigsqcup\limits_{i=1}^{m_1}(b_i x_{i,+})\right)\bigsqcup\left(\bigsqcup\limits_{i=1}^{m_1}(\wi b_i y_{i,+})\right)\bigsqcup\left(\bigsqcup\limits_{i=1}^{m_1}(\bar b_i z_{i,+})\right), \]
where $b_i, \wi b_i, \bar b_i\in\{0,1\}$ with $b_i+ \wi b_i+ \bar b_i\ls 1$ for $i\in \Lambda_1$,   and $\sum_{i=1}^{m_1}(b_i+ \wi b_i+ \bar b_i)\gl 1$.
Among these critical points, there exists at least $m_1$ sign-changing solutions, $2^{m_1}-1$ positive critical points and $2^{m_1}-1$ negative critical points.

Next, we show that if we assume  that $m_1=1$, $m_2=0$ and $K\bk (P\cup (-P))$ is a finite number set,  then \( J \) has a critical point \( z_0 \in X \setminus (\overline{D}_1^X \cup \overline{D}_2^X) \) satisfying \( C_2(J, z_0) \neq 0 \). By Lemma 3.11 and the excision property of homology theory, we see that for $\va>0$ small enough,
\[0\neq H_2(Q^{c_1+\varepsilon}, Q^{c_1-\varepsilon})\cong H_2(Q^{c_1}, Q^{c_1}_*)\cong H_2(S^{c_1}, S^{c_1}\bk K_{c_1}), \]
where $S^{c_1}:=(J^{-1}(c_1)\cap K)\bk (P\cup (-P))$.
 Now suppose that
\[
S^{c_1} \cap K = \{ u_1^\ast, u_2^\ast, \dots, u_m^\ast \}.
\]
Then we have for \( r > 0 \) small enough
\[
\begin{aligned}
H_2(S^{c_1}, S^{c_1} \setminus K_{c_1})
&\cong H_2\left(
S^{c_1} \cap \left( \bigcup_{i=1}^m B(u_i^\ast, r) \right),
(S^{c_1} \setminus K_{c_1}) \cap \left( \bigcup_{i=1}^m B(u_i^\ast, r) \right)
\right) \\
&\cong \bigoplus_{i=1}^m H_2\left(
J^{c_1} \cap B(u_i^\ast, r),
(J^{c_1} \cap B(u_i^\ast, r)) \setminus \{u_i^\ast\}
\right) \\
&\cong \bigoplus_{i=1}^m C_2(J, u_i^\ast),
\end{aligned}
\]
where \( C_2(J, u_i^\ast) \) denotes the local homology group at \( u_i^\ast \). Thus, there exists $z_0\in K^*_{c_1}$ with $C_2(J, z_0)\neq 0$.

If the conditions of 2) hold,  an  argument similar to Theorem 3.1 yields that for each $j\in\Lambda_2$, $J$ has at least one  critical point $w_j\in D_{j,-}^{dd}$. Hence, in this case $J$ admits at least $2^{m_2}-1$  critical points with each of theses critical points has the form  $\bigsqcup\limits_{j=1}^{m_2}(d_j w_j)$, where $d_j\in \{0,1\}$ for $j\in\Lambda_2$  with $0\neq (d_1, d_2,\cdots, d_{m_2})\in \mb R^{m_2}$.

Assume that the conditions of 3) hold.  Then  in a similar argument as in  Theorem 3.1 yields that $J$ possesses at least $ 2^{2m_1+m_2}-1$ distinct nontrivial critical points such that each of these  critical points has the one of the three forms:
\[\left(\bigsqcup\limits_{i=1}^{m_1}(b_i x_{i,+})\right)\bigsqcup\left(\bigsqcup\limits_{i=1}^{m_1}(\wi b_i y_{i,+})\right)\left(\bigsqcup\limits_{i=1}^{m_1}(\bar b_i z_{i,+})\right)\left(\bigsqcup\limits_{j=1}^{m_2}(d_j w_{j})\right), \]
where $b_i, \wi b_i, \bar b_i\in\{0,1\}$ with $b_i+ \wi b_i+ \bar b_i\ls 1$ for $i\in \Lambda_1$,  $d_j\in \{0,1\}$ for $j\in \Lambda_2$ and $\sum_{i=1}^{m_1}(b_i+ \wi b_i+ \bar b_i)+\sum_{j=1}^{m_2}d_j\gl 1$. Among these critical points there exists at least $m_1$ sign-changing critical points, $2^{m_1}-1$ positive critical points and $2^{m_1}-1$ negative critical points.
\end{proof}

\begin{remark} In the proof above  we  use some ideas from \cite{s20}. Clearly, under the condition of the space decomposition in  (H$_1)$,  one can discuss the existence of infinitely many  critical points for symmetric functionals.  We leave this to the interested readers.\end{remark}

\begin{remark} In Theorem 3.1 and 3.2,  we require that $J$ satisfies  the (P.S.) condition on each $D_{i,+}^{dd}$ and $D_{j,-}^{dd}$. This condition is weaker than requiring $J$ to satisfy the (P.S.) condition on the entire space $X$, and it is also easier to verify. \end{remark}

\begin{remark}Clearly, the requirement that $g(s)=o(s)$ as $s\ri 0^+$ in (H$_5)$ can be weaken.\end{remark}

\begin{definition}
Let \(E\) be a Riesz space, and let \(E_1, E_2, \dots, E_m\) be linear subspaces of \(E\). The subspaces are said to be \textit{pairwise lattice disjoint} if for any \(i \neq j\), every \(x \in E_i\) and \(y \in E_j\) satisfy \(x \perp y\) (i.e., \(|x|\wedge|y| = 0\)).

The direct sum \(D = \bigoplus_{i=1}^m E_i\) is called \textit{complementable by a lattice-disjoint subspace} in \(E\) if there exists a linear subspace \(E_0 \subseteq E\) such that \(E_0\) is lattice disjoint from each \(E_i\) (\(i=1,\dots,m\)) and
\[
E = E_0 \oplus D.
\]
\end{definition}

\begin{remark} Using the above concepts, we can relax the corresponding space decomposition requirement in (H\(_1)\) to the condition that $\bigoplus\limits_{i=1}^{m_1}G^{dd}_{i,+}\oplus\bigoplus\limits_{j=1}^{m_2}G^{dd}_{j,-}$ is complementable by a lattice-disjoint subspace in $E$. This is useful in the study of boundary value problems for elliptic equations; see Remark 4.1 below.    \end{remark}

\begin{remark}In Theorem 3.2, we obtain the existence of $z_i$ by means of homology group computation.
It should be noted that this result can also be proved by applying the four critical points theorem given in Reference \cite{s19}; see Theorem 3.3 in \cite{s19}.\end{remark}

\section{ Applications to Elliptic Boundary Value Problems}

\noindent Consider the elliptic boundary value problem
\[\left\{
\be{ll}
-\Delta u =a^+(x) g_1(x, u)-a^-(x)(g_2(x, u)-h(x, u)), & x \in \Omega, \\
u|_{\partial \Omega} = 0,
\en\right.\no(4.1)
\]
where $\Omega \subset \mathbb{R}^N$ is a bounded domain with a $C^1$ boundary,  $g_1, g_2, h \in C(\ov\Omega\ti \mathbb{R}, \mathbb{R})$, $a \in C(\overline{\Omega}, \mathbb{R})$, $a^+(x)=\max\{a(x), 0\}$ and $a^-(x)=\max\{-a(x), 0\}$ for $x\in \ov\Om$.
\medskip

First, we give some hypotheses:
\begin{itemize}
\item[(A$_1)$]\ $a\in C(\ov \Omega)$,  $\Omega^+:=\{x\in \Omega: a(x)>0\}$ has $m_1$ connected components,  $\Omega^-:=\{x\in \Omega: a(x)<0\}$ has $m_2$ connected components, and $| \ov{\Omega}^+\cap \ov{\Omega}^-|=0$;

\item[(A$_2)$]\ Set $2^*=+\infty$ if $N=1,2$ and $2^*=2N/(N-2)$.  There exist $b > 0$, $2 < p < 2^* $, such that
\[
\max\{|g_1(x,t)|, |g_2(x,t)|, |h(x,t)|\} \ls b(1 + |t|^{p-1})\   \mbox{ for}\ x\in\Omega\quad \mbox{and}\  t \in \mathbb{R};
\]

\item[(A$_3)$]\  For each $t\in\mb R$ and $x\in\Om$, $G_2(x,t)\gl 0$, $H(x,t)\gl 0$, and there exist $R > 0$, $0 < \theta < \frac{1}{2}$, such that for $x\in \Om\ \mbox{and}\ |t| \gl R$,
 \[0 < G_1(x,t) \ls \theta t g_1(x, t), \]
where $$G_1(x, t) = \int_{0}^{t} g_1(x, s) ds,  G_2(x, t) = \int_{0}^{t} g_2(x, s) ds,  H(x, t) = \int_{0}^{t} h(x, s) ds;$$

\item[(A$_4)$]\  For some $2^*>\beta_1>2$, $$\limsup_{t \to 0^+} \frac{g_1(x, t)}{t^{\beta_1}} =0\ \mbox{uniformly in}\  x\in\ov\Om;$$

\item[(A$_5)$]\  For some $2^*>\beta_2>2$ and $\al_0\gl 0$, $$\lim\limits_{t \ri 0^+} \frac{g_2(x,t)}{|t|^{\beta_2-2}t} =\al_0 \ \mbox{uniformly in}\ x\in\ov\Om;$$

\item[(A$_6)$]\ For some $\beta_3, \bar\beta_3\in (1,2)$  and $\al_1, \bar\al_1>0$, $$\lim_{t \to 0^+} \frac{h(x, t)}{|t|^{\beta_3-2}t} =\al_1\ \mbox{and}\ \lim_{t \to \infty} \frac{h(x, t)}{|t|^{\bar\beta_3-2}t} =\bar\al_1\ \mbox{uniformly in}\ x\in\ov\Om;$$

\item[(A$_7)$]\ $g_1(x, u)u\gl 0$ for each $x\in \Om$ and  $u\in\mb R$.
\end{itemize}

The  $X:=H_0^1(\Omega)$ is endowed with the inner product and associated norm
\[
(u, v) = \int_{\Omega} \nabla u \cdot \nabla v  dx, \quad u, v \in H_0^1(\Omega),
\]
\[
\|u\| = \left( \int_{\Omega} |\nabla u|^2  dx \right)^{1/2}.
\]
The  $E:=L^2(\Omega)$ is endowed with the inner product and associated norm
\[
\langle u, v\rangle = \int_{\Omega}  u \cdot  v  dx, \quad u, v \in H_0^1(\Omega),
\]
\[
\|u\|_2 = \left( \int_{\Omega} | u|^2  dx \right)^{1/2}.
\]
Let $P_1:=\{u\in L^2(\Om):u(x)\gl 0 \ \text{a.e. in}\ \Om\}$ and $P=P_1\cap X$. Then $P_1$ is a cone in $E$, and $P$ a  cone in $X$, respectively. Define the ordering $\ls$ in $E$ by $u\ls v$ if and only if $v-u\in P_1$. It is easy to see that $E$ is a Banach lattice and $X$ is a Riesz space but no a Banach lattice. Assume that the $m_1$ connected components of $\{x\in \Om: a(x)>0\}$  are $\Om_1^+, \Om_2^+,\cdots, \Om_{m_1}^+$,  and the $m_2$ connected components of $\{x\in \Om: a(x)<0\}$  are $\Om_1^-, \Om_2^-,\cdots, \Om_{m_2}^-$. For $i\in\Lambda_1$ and $j\in\Lambda_2$, let $a_i^+$  and $a_j^-$ be defined by
\[a_i^+(x)=\left\{\be{ll} a(x), &x\in \Om_i^+;\\
0, &x\in \Om\bk \Om_i^+\en\right.\  \mbox{and}\
a_j^-(x)=\left\{\be{ll} a(x), &x\in \Om_j^-;\\
0, &x\in \Om\bk \Om_j^-.\en\right.\]
 According to (A$_1)$, we have
\[a^+=\bigsqcup_{i=1}^{m_1} a^+_i, a^-=\bigsqcup_{j=1}^{m_2} a^-_j.\]
Correspondingly, $E$ has a decomposition:
\[E=\bigoplus\limits_{i=1}^{m_1}G^{dd}_{i,+}\oplus\bigoplus\limits_{j=1}^{m_2}G^{dd}_{j,-},\]
and $X$ has a decomposition:
\[X=\bigoplus\limits_{i=1}^{m_1}D^{dd}_{i,+}\oplus\bigoplus\limits_{j=1}^{m_2}D^{dd}_{j,-}.\]

For  $u, v\in E$ with $u\bot v$,
we have $$|{\mbox{supp}}\ u\cap {\mbox{supp}}\ v|=0,\no(4.2)$$
and so $$\int_\Om u\cdot v dx=0.$$
Hence, disjointness and orthogonality in $E$ are compatible.
\medskip

Let us denote as in \cite{s11} the linear sapce of $k$ times weakly differentiable functions by $W^k(\Om)$. According to Lemma 7.7 in \cite{s11} we have the following Lemma 4.1.
\begin{lemma}\  Let $u\in W^1(\Omega)$. then $Du=0$ a.e. on any set where $u$ is constant.
\end{lemma}

For  $u, v\in X(\subset W^1(\Om))$ with $u\bot v$, it follows from Lemma 4.1 that   $$Du=0\ \mbox{a.e. on}\ \Om\bk\mbox{supp}\ u,\  \mbox{and}\  Dv=0\ \mbox{a.e. on}\ \Om\bk\mbox{supp}\ v.$$ By (4.2), we have $$\big|{\mbox{supp}}\ |Du|\cap {\mbox{supp}}\ |Dv|\big|=0, $$ and so $(u,v)=0$. Hence, disjointness and orthogonality in $X$ are compatible.

\medskip

Let $J:X\ri\mb R$ be the functional corresponding to (4.1):
$$J(u)=\frac{1}{2}\|u\|^2-\int_\Om \left(a^+(x) G_1(x,u)-a^-(x)G_2(x,u)+a^-H(x,u)\right)dx.$$
It follows that
$$(J'(u), v)=(u,v)-\big\langle a^+(x) g_1(x, u)-a^-(x)g_2(x, u)+a^-(x)h(x,u), v\big\rangle.$$
Let $\mathbf{f}:X\ri E^*$ be defined by
$$\mathbf{f}(u)(x)=a^+(x) g_1(x, u)-a^-(x)(g_2(x, u)-h(x,u))\ \mbox{for}\ x\in \Om.$$ Then we have $(J'(u), v)=(u,v)-\langle \mathbf{f}(u), v\rangle$ for each $u,v\in X$. It follows from (A$_4)$, (A$_5)$ and (A$_6)$ that  $$g_1(x, 0)=g_2(x, 0)=h(x, 0)=0\ \mbox{for}\  x\in \ov\Om.$$ Then, it is easy to see that $\mathbf{f}$ is local and orthogonally additive.

 Clearly, $\mathbf{f}:L^{2^*}(\Om)\to E$ is continuous. By Sobolev Embedding Theorem, $i: X\to  L^{2^*}(\Om)$ is continuous. Thus, $\mathbf{f}=\mathbf{f}\circ i:X\to E$ is continuous.
Hence, (H$_1)$ holds.
\medskip

\begin{lemma}\  Condition (H\(_{2})\) holds.\end{lemma}
\begin{proof}\ It follows from (A$_3)$ that for some $c_2,c_3>0$,
$$G_1(x, z)\gl c_2|z|^{\frac{1}{\theta}}-c_3\ \text{for}\ x \in\Om \ \text{and}\ z\in\mb R.$$
For $i\in\{1,2,\cdots, m_1\}$, $u\in D_{i,+}^{dd}\bk\{0\}$ and $t>0$,   we have $$a^+(x)G_1(x, tu(x))=0\ \mbox{a.e.}\ x\in \Om\bk\Om_{i}^+,$$
and $$a^-(x)G_2(x, tu(x))=a^-(x)H(x, tu(x))=0\ \mbox{  a.e.}\ x\in\Om.$$ Consequently, we have for $u\in D_{i,+}^{dd}$,
$$
\be{ll}J(tu)&=\frac{t^2}{2}\|u\|^2-\di\int_{\Om_{i}^+} a^+(x) G_1(x, tu)dx\\
&\ls \frac{t^2}{2}\|u\|^2-c_2t^{\frac{1}{\theta}}\di\int_{\Om_{i}^+} a^+(x) |u|^{\frac{1}{\theta}}dx+c_3|\Om|\ri-\infty
\en
$$
 as $t\ri+\infty$. This implies that (H$_2)$ holds. \end{proof}
\medskip
\begin{lemma}\  Condition (H\(_{3})\) holds.\end{lemma}
\begin{proof}It follows from (A$_5)$ and (A$_6)$ that for some $d_0, d_1>0$ small enough such that  \[g_2(x, u)\ls \frac{3(\al_0+d_1)}{2}|u|^{\beta_2-2} u \ \mbox{and}\ h(x,u)\gl \frac{\al_1}{2}|u|^{\beta_3-2} u \] for  $0\ls u\ls d_0$.  Then,  for some  $u_0\in D_{j,-}^{dd}\bk\{0\}$ and $t>0$ such that $\|t u_0\|_\infty<d_0$, we have$$\be{ll}J(tu_0)&=\frac{t^2}{2}\|tu_0\|^2+\di\int_{\Om_{j}^{-}} a^-(x) G_2(x,tu_0(x))dx-\di\int_{\Om_{j}^{-}} a^-(x) H(x, tu_0(x))dx\\
&\ls \frac{t^2}{2}\|u_0\|^2+\frac{3(\al_0+d_1)t^{\beta_2}}{2\beta_2}\di\int_{\Om_{j}^{-}} a^-(x) |u_0(x)|^{\beta_2}dx-\frac{\al_1t^{\beta_3}}{2\beta_3}\di\int_{\Om_{j}^{-}} a^-(x) |u_0(x)|^{\beta_3}dx.\en$$
Note that $\beta_2>2$ and $\beta_3\in (1,2)$. Then we have $J(t_0 u_0)<0$  for $t_0>0$ small enough. Hence,
$$\inf\limits_{u\in D_{j-}^{dd}}J(u)<0.$$

It follows from (A$_5)$ and (A$_6)$ that for  some $c_4>0$ large enough and $u\in\mb R$, \[h(x,u)\ls \frac{3\bar\al_1}{2}|u|^{\bar\beta_3-2} u+c_4. \]
Then, by Sobolev embedding theorem  we have for each  $u\in D_{j,-}^{dd}$ and some $c_5,c_6>0$, $$\be{ll}J(u)&=\frac{1}{2}\|u\|^2+\di\int_{\Om_{j}^{-}} a^-(x) G_2(x,u(x))dx-\di\int_{\Om_{j}^{-}} a^-(x) H(x, u_0(x))dx\\
&\gl \frac{1}{2}\|u\|^2-\frac{\bar\al_1}{2\bar\beta_3}\di\int_{\Om_{j}^{-}} a^-(x) |u_0(x)|^{\bar\beta_3}dx-c_5\\
&\gl \frac{1}{2}\|u\|^2-\frac{\bar\al_1 \|a^-\|_\infty}{2\bar\beta_3}\di\int_{\Om}  |u_0(x)|^{\bar\beta_3}dx-c_5\\
&\gl \frac{1}{2}\|u\|^2-\frac{c_6\bar\al_1 \|a^-\|_\infty}{2\bar\beta_3}\|u\|^{\bar\beta_3}-c_5.\en$$
Note that $\bar\beta_3\in (1,2)$. Then we have $J(u)\ri+\infty$  as $\|u\|\ri +\infty$. Hence, (H\(_{3})\) holds.
\end{proof}

  Write the standard norm in $L^p(\Om)$ as $\|\cdot\|_p$ for $p\gl 1$.

\begin{lemma}\  Condition (H\(_{4})\) and (H\(_{5})\) hold.\end{lemma}
\begin{proof}\ For each $u\in D_{i,+}^{dd}\cap P$, $\mathbf{f}(u)(x)=a^+(x)g_1(x, u)\gl 0$  \ \mbox{for}\  $x\in \Om$, and so
$$\langle \mathbf{f}(u(x)), v(x)\rangle\gl 0$$
for each  $v\in P_1$, which implies that $\mathbf{f}(D_{i,+}^{dd}\cap P)\subset P_1^*$. Hence, (H$_4)$ holds.

  It follows from $(A_{2})$ and $(A_{4})$ that for some $c_7>0$, \[|g_1(x, u)|\ls c_7(|u|^{\beta_1-1}+|u|^{p-1})\ \mbox{for}\ x\in\Omega\ \mbox{and}\ u\in\mathbb{R}.\]
By using H\"{o}lder inequality we have $$\langle \mathbf{f}(u),w\rangle\ls c_7(\|u\|_{\beta_1}^{\beta_1-1}\|w\|_{\beta_1}+\|u\|_p^{p-1}\|w\|_p).$$
Let $q=\max\{\beta_1, p\}>1$ and $Z=L^q(\Om)$. Then we have
\[X\hookrightarrow Z,\  Z\hookrightarrow L^{\beta_1}(\Om),\  Z\hookrightarrow L^{p}(\Om).\]
So, we have,
$$\langle\mathbf{f}(u),w\rangle\ls  c_8(\|u\|_{q}^{\beta_1-1}+\|u\|_{q}^{p-1})\|w\|_q$$
for some $c_8>0$. Let $g(s)=c_8(s^{\beta_1-1}+ s^{p-1})$ for any $s\in \mb R_+$. Then the condition (H\(_{5})\) holds.
\end{proof}

Note that for each $i=1,2,\cdots,m_1$, \[J(u)=\frac{1}{2}\|u\|^2-\di\int_{\Om_{i}^+} a^+(x) G_1(x, u)dx\ \mbox{for}\ u\in D_{i,+}^{dd}.\]

By usual way we can show that $J$ satisfies the (P.S.) condition on each $D_{i,+}^{dd}$. Since $J$ is coercive on each $D_{j,-}^{dd}$, we can easily see that $J$ satisfies (P.S.) condition on $D_{j,-}^{dd}$.

Since we can construct $2^{2m_1}-1$  nontrivial solutions whose support set contained in $\Om^+$ ,
 $2^{m_1}-1$ positive solutions and  $2^{m_1}-1$ negative solutions  whose support set contained in ${\Om}^+$.
Hence, the minimal number of  sign-changing solutions whose support set contained in $\Om^+$ is
\[
2^{2m_1} - 2^{m_1+1} + 1.
\]
Therefore, the minimal number of  sign-changing solutions defined on $\Om$ is
\[
\big(2^{2m_1} - 2^{m_1+1} + 1\big)\cdot 2^{m_2}.
\]

According to Theorem 3.2 and Lemma 4.1$\sim$ 4.4, now we have the following result.

\begin{theorem}\ Assume  that  the conditions (A\(_{1})\sim\)(A\(_{7})\) hold. Then $(4.1)$ has at least $2^{2m_1+m_2}-1$ solutions. Among them, there exist at least $2^{m_1}-1$ positive solutions, $2^{m_1}-1$ negative solutions and $(2^{2m_1}-2^{m_1+1}+1)\cdot 2^{m_2}$ sign-changing solutions.
\end{theorem}

\begin{remark} Set
\[D=\{g\in E: \mbox{supp}\ g\subset_1 \ov{\Omega}^+\cap \ov{\Omega}^-\}\]if $| \ov{\Omega}^+\cap \ov{\Omega}^-|\neq 0$. Then we  have
\[E=\bigoplus\limits_{i=1}^{m_1}G^{dd}_{i,+}\oplus\bigoplus\limits_{j=1}^{m_2}G^{dd}_{j,-}\oplus D.\]
 This implies that $\bigoplus\limits_{i=1}^{m_1}G^{dd}_{i,+}\oplus\bigoplus\limits_{j=1}^{m_2}G^{dd}_{j,-}$ is complementable by a lattice-disjoint subspace in \(E\). Thus, in the case of $| \ov{\Omega}^+\cap \ov{\Omega}^-|\neq 0$,  we can also give some existence results for solutions of (4.1) by making use of Theorem 3.2.
\end{remark}

\begin{remark}Here we present an application of Theorem 3.2 to the study of indefinite elliptic equations.
Theorem 3.2 admits extensive applications in  other indefinite problems. Interested readers may derive some corresponding results along  this line.\end{remark}


\begin{thebibliography}{99}
\vspace{0.5em}
\bibitem{s1} K.C. Chang, M.Y. Jiang, Morse theory for indefinite nonlinear elliptic problems.
Ann. Inst. H. Poincar\'{e} C Anal. Non Lin\'{e}aire 26 (2009), no. 1, 139-158.
\bibitem{s2} K.C. Chang,  M.Y. Jiang, Dirichlet problems with indefinite nonlinearities, Cacl. Var. 20 (2004) 257-282.
\bibitem{s3}N. Ackermann, T. Bartsch, P. Kaplicky, P. Quittner, A priori bounds, nodal equilibria and connecting orbits in indefinite superlinear parabolic problems, Trans. Amer. Math. Soc., 360(2008), no. 7, 3493-3539.
\bibitem{s4}H. Berestycki, I. Capuzzo-Dolcetta, L. Nirenberg, Superlinear indefinite elliptic problem and nonlinear Liouville theorems, Topol. Methods Nonlinear Anal. 4 (1994) 59-78.
\bibitem{s5}H. Berestycki, I. Capuzzo-Dolcetta, L. Nirenberg, Variational methods for indefinite superlinear homogeneous elliptic problems, NoDEA 2 (1995) 533-572.
\bibitem{s6}D. Costa, H. Tehrani, Existence of positive solutions for a class of indefinite elliptic problems in $\mb R^N$, Calc. Var. Partial Differential Equations 13 (2001) 159-189.
\bibitem{s7}D. Costa, H. Tehrani, Existence and multiplicity results for a class of Schr\"{o}dinger equations with indefinite nonlinearities, Adv. Differential Equations 8 (2003) 1319-1340.
\bibitem{s8}H. Amann, J. L\'{o}pez-G\'{o}mez, A priori bounds and multiple solutions for superlinear indefinite nonlinear elliptic equations, J. Differential Equations 146 (1998) 336-374.
\bibitem{s9}S. Alama, M. Del Pino, Solutions of elliptic equations with indefinite nonlinearities via Morse theory and linking, Ann. Inst. H. Poincar\'{e} Anal. Nonlin\'{e}aire 13 (1996) 95-115.
\bibitem{s10}S. Alama, G. Tarantello, Elliptic problems with nonlinearities indefinite in sign, J. Funct. Anal. 141 (1996) 159-215.
\bibitem{s11}D. Gilbarg, N.S. Trudinger, Elliptic Partial Differential Equations of Second Order, Springer-Verlag, 2001.
\bibitem{s12} Luxemburg W. A. J., Zaanen A. C., Riesz spaces. Vol. I. North-Holland Math. Library
North-Holland Publishing Co., Amsterdam-London; American Elsevier Publishing Co., Inc., New York, 1971, xi+514 pp.
\bibitem{s13}Abasov Nariman; Pliev Marat, Disjointness-preserving orthogonally additive operators in vector lattices.
Banach J. Math. Anal. 12 (2018), no. 3, 730-750.
\bibitem{s14}Mykhaylyuk Volodymyr, Pliev Marat; Popov Mikhail, The lateral order on Riesz spaces and orthogonally additive operators.
Positivity 25 (2021), no. 2, 291-327.
\bibitem{s15}J. Sun, X. Liu, Computation of topological degree in ordered Banach spaces with lattice structure and its application to superlinear differential equations, J. Math. Anal. Appl. 348 (2008), no. 2, 927-937.
\bibitem{s16}J. Sun, X. Xu, Positive solutions of semi-positone nonlinear operator equations in Banach spaces with lattice structure.
Positivity 17 (2013), no. 4, 995-1007.
\bibitem{s17}J. Sun, X. Xu, Positive solutions of non-positone nonlinear operator equations and their applications.
Acta Math. Sinica (Chinese Ser.) 55 (2012), no. 1, 55-64.
\bibitem{s18}Perera K. and Schechter M., Topic in Critical Point Theory, Cambridge University Press, 2013.
\bibitem{s19} Z. Liu, J. Sun, Invariant sets of descending flow in critical point theory with applications to nonlinear differential equations, J. Differential Equations 172 (2001) 257-299.
\bibitem{s20}T. Bartsch and Z. Liu, Location and critical groups of critical points in Banach spaces and applications to nonlinear eigenvalue problems. Adv. Differential equations, 9(5-6): 645-676,2004.
\bibitem{s21}Heinz H. Bauschke,
Patrick L. Combettes, Convex Analysis and Monotone Operator Theory in Hilbert Spaces (Second edition),  Springer, New York (2016).
\bibitem{s22}Deimling K., Ordinary Differential Equations in Banach Spaces, Springer-Verlag, Berlin Heidelberg New York 1977.
\bibitem{s23}Z. Zhang, K. Perera, Sign changing solutions of Kirchhoff type problems via invariant sets of descent flow, J. Math. Anal. Appl. 317 (2006), no. 2, 456-463.
\bibitem{s24}Z. Zhang,  Chen Jianqing, S. Li,
Construction of pseudo-gradient vector field and sign-changing multiple solutions involving $p$-Laplacian, J. Differential Equations, 201(2004), no.2, 287-303.
\bibitem{s25}T. Bartsch and Z.Liu,
On a superlinear elliptic $p$-Laplacian equation. J. Differential Equations, 198(2004), no.1, 149-175.
\bibitem{s26}T. Bartsch and S. Li, Critical point theory for asymptotically quadratic functionals and applications to problems with resonance, Nonlinear Anal., 28(3): 419-441,1997.
\bibitem{s27}Z. Q. Wang, On a superlinear elliptic equation,    Ann. Inst. H. Poincar\'{e} C Anal. Non Lin\'{e}aire 26 (1991), no. 8, 43-57.
\bibitem{s28} K.C. Chang, Methods in Nonlinear Analysis, Springer-Verlag, Berlin Heidelberg, 2005.
\end{thebibliography}
\end{document}